\documentclass[12pt]{amsart}
\usepackage{graphicx} 
\usepackage{xcolor}
\usepackage{amsmath}
\usepackage{amsfonts}
\usepackage{mathtools}
\usepackage{amsthm}
\usepackage{mathabx}
\usepackage{comment}
\usepackage{hyperref}
\usepackage{quiver}
\usepackage{tikz}
\usepackage[left=1.5in,right=1.5in,top=1.4in,bottom=1.4in]{geometry}
\usepackage[foot]{amsaddr}

\usepackage[backend=biber, style=numeric, sorting=nyt, maxcitenames=10, minbibnames=10, maxbibnames=10, doi=false, isbn=false]{biblatex}
\renewbibmacro{in:}{%
  \ifentrytype{article}{}{\printtext{\bibstring{in}\intitlepunct}}} 
\newtheorem{main}{Theorem}

\newtheorem{maincor}[main]{Corollary}

\newtheorem{theorem}{Theorem}[section]
\newtheorem{corollary}[theorem]{Corollary}
\newtheorem{example}[theorem]{Example}
\newtheorem{lemma}[theorem]{Lemma}
\newtheorem{fact}[theorem]{Fact}
\newtheorem{proposition}[theorem]{Proposition}
\newtheorem{definition}[theorem]{Definition}
\newtheorem{remark}[theorem]{Remark}
\newtheorem{question}[theorem]{Question}
\newtheorem{construction}[theorem]{Construction}
\newtheorem*{corollary*}{Corollary C'}

\DeclareMathOperator{\EFD}{EFD}
\DeclareMathOperator{\qr}{qr}

\DeclareMathOperator{\Mod}{Mod}
\DeclareMathOperator{\bfMod}{\bf{Mod}}

\DeclareMathOperator{\WO}{WO}

\DeclareMathOperator{\id}{id}

\DeclareMathOperator{\Ell}{\mathbf{Ell}}

\DeclareMathOperator{\Mor}{Mor}

\DeclareMathOperator{\Ob}{Ob}
\DeclareMathOperator{\image}{im}
\DeclareMathOperator{\kernel}{ker}
\DeclareMathOperator{\ev}{ev}

\DeclareMathOperator{\Aff}{Aff}

\DeclareMathOperator{\inv}{inv}

\DeclareMathOperator{\Pairing}{Pairing}
\DeclareMathOperator{\States}{States}
\DeclareMathOperator{\proj}{proj}

\DeclareMathOperator{\LO}{LO}

\DeclareMathOperator{\EF}{EF}

\newcommand{\Cstar}{\mathbf{C^*}} 
\newcommand{\theory}{T}
\newcommand{\fancyK}{\mathcal{K}}
\newcommand{\Lang}{L}
\newcommand{\cZ}{\mathcal{Z}}
\newcommand{\Kconv}{\mathbf{K}_{\textnormal{conv}}}

\def\dotminussym#1#2{%
  \setbox0=\hbox{$\m@th#1-$}%
  \kern.5\wd0%
  \hbox to 0pt{\hss\hbox{$\m@th#1-$}\hss}%
  \raise.6\ht0\hbox to 0pt{\hss$\m@th#1.$\hss}%
  \kern.5\wd0}

\NewDocumentCommand{\notrel}{m}
 {%
  \ooalign{$#1$\cr\noalign{\kern-.2ex}\hidewidth$/$\hidewidth\cr}
 }

\DeclarePairedDelimiter{\ip}{\langle}{\rangle}
\DeclarePairedDelimiter{\norm}{\lVert}{\rVert}

\title{A Game-Theoretic Unital Classification Theorem for $C^*$-Algebras}
\author{Jennifer Pi}
    \address{Mathematical Institute, University of Oxford, Oxford, United Kingdom}
    \email{jennifer.pi@maths.ox.ac.uk}
\author{Micha\l \, Szachniewicz}
    \address{Institute for Advanced Study, 1 Einstein Drive, Princeton, New Jersey, 08540, USA}
    \email{michals@ias.edu}
    
\author{Mira Tartarotti}
    \address{Mathematical Institute, University of Oxford, Oxford, United Kingdom}
    \email{tartarotti@maths.ox.ac.uk}
\date{}

\begin{document}

\begin{abstract}
    We study the complexity of the $KK$-equivalence relation on unital $C^*$-algebras, in the sense of descriptive set theory. 
    We prove that $KK$-equivalence is analytic, which in turn shows that the set of separable $C^*$-algebras satisfying the UCT is analytic. This allows us to prove a game-theoretic refinement of the unital classification theorem: there is a transfer of strategies between Ehrenfeucht-Fra\"iss\'e games (of various lengths) on classifiable $C^*$-algebras and their invariants.
\end{abstract}

\maketitle

\tableofcontents

\section{Introduction}

Classification of $C^*$-algebras is a long-running program which seeks to identify invariants which classify, up to isomorphism, a large class of separable $C^*$-algebras. 
This endeavor is known as the Elliott classification program, and seeks to prove $C^*$-algebraic analogues of the celebrated and influential structure and classification results for amenable von Neumann factors. (A \textit{factor} is a von Neumann algebra with trivial center, which serves as a basic irreducible building block.)
The classification program in the case of von Neumann algebras shows that each amenable factor is classified up to isomorphism only by its type (which describes the structure of projections) and a numerical invariant.

$C^*$-classification developed with close relations to these structural results of von Neumann algebras.
Indeed, the very first classification result for simple $C^*$-algebras is a precise analogue to Murray and von Neumann's characterization of the hyperfinite II$_1$ factor \cite{murray1943rings}, given by Glimm's classification of uniformly hyperfinite $C^*$-algebras \cite{glimm1960}.
The modern-day program mirrors Connes' groundbreaking work classifying amenable von Neumann algebras \cite{connes1976classification}, and has been an ongoing endeavor involving many mathematicians since the 1970s. This began with Elliott's result classifying approximately finite-dimensional (AF) algebras using a relatively simple invariant: their ordered $K_0$ groups \cite{elliott1976AF}.
The classification program continued with the Kirchberg-Phillips theorem \cite{kirchberg1995classif, phillips2000classif}, classifying a large class of algebras by their operator-algebraic $K$-theory. 
Finally, the general version of the unital classification theorem combines the Kirchberg-Phillips theorem with a stably finite classification theorem which was proved combining \cite{EGLN-classification, GLN-classif-1, GLN-classif-2, TWW-quasidiagonality}, and the many papers these rely on. 
We use the version from the new abstract approach to the stably finite classification \cite[Theorem A]{CGSTW-new-classificationI-2023}.

\begin{theorem}\label{intro thm: classification of C*-algebras}
Unital simple separable nuclear $\cZ$-stable $C^*$-algebras satisfying Rosenberg and Schochet’s universal coefficient theorem (UCT) are classified by the invariant $KT_u$ consisting of $K$-theory and traces.
\end{theorem}

We hereafter refer to this class of $C^*$-algebras as \textit{classifiable algebras}. The class directly parallels Connes' classification of amenable factors; indeed, the assumptions of unitality, simplicity, and separability are immediate in the case of von Neumann factors, and nuclearity is equivalent to amenability in the von Neumann factor setting.
The term ``$\cZ$-stable" refers to tensorial absorption of a particular $C^*$-algebra called the Jiang-Su algebra \cite{MR1680321}, denoted by $\cZ$.
This assumption is crucial since the $K$-theory of an algebra $A$ and $A \otimes \cZ$ are identical. 
Finally, the assumption of the UCT allows one to connect Kasparov's $KK$-theory \cite{Kasparov1981-origKK} with $K$-theory, which is necessary to phrase the invariant in terms of $K$-theory. 
Indeed, the aforementioned Kirchberg-Phillips theorem shows that the purely infinite classifiable algebras are isomorphic if and only if they are $KK$-equivalent. 
We will return to discuss the UCT assumption shortly after stating Theorem \ref{theorem: intro uniform Borelness of games}.
\\

Theorem \ref{intro thm: classification of C*-algebras} says that two classifiable $C^*$-algebras are isomorphic if and only if their invariants are isomorphic.
It is sensible to ask whether the classification theorem can be extended to coarser equivalence relations between classifiable $C^*$-algebras, respectively between their invariants. We answer precisely this question in the positive, for a family of equivalence relations arising from the infinitary logic $L_{\omega_1, \omega}$.
When $L$ is a discrete first-order language, the set of $L_{\omega_1, \omega}$-formulas is obtained by closing the set of first-order $L$-formulas under countable conjunctions and disjunctions. The \emph{quantifier-rank} of an $L_{\omega_1, \omega}$-formula $\varphi$ is a countable ordinal $\qr(\varphi)<\omega_1$ corresponding roughly to the maximal depth of nested quantifier-chains in $\varphi$. Setting $A\equiv_\alpha B$ whenever $A$ and $B$ are $L$-structures that agree on $L_{\omega_1, \omega}$-sentences of quantifier rank at most $\alpha$ yields a sequence $(\equiv_\alpha)_{\alpha<\omega_1}$ of refining equivalence relations. Scott's isomorphism theorem \cite{Scott1965-SCOLWD} states that the intersection of this sequence coincides with isomorphism, that is, countable $L$-structures $A$ and $B$ are isomorphic (written $A\cong B$) if and only if $A\equiv_\alpha B$ for every $\alpha<\omega_1$.
To study metric structures such as $C^*$-algebras in the context of (infinitary) logic, we use \emph{metric model theory} as developed by Ben Yaacov, Berenstein, Henson and Usvyatsov  \cite{MR2436146}. The corresponding generalization of $L_{\omega_1, \omega}$ due to Ben Yaacov, Doucha, Nies and Tsankov \cite{yaacov2017metric} serves as the basic framework of this paper, providing in particular a metric version of the equivalence relations $(\equiv_\alpha)_{\alpha<\omega_1}$ and of Scott's isomorphism theorem.

In \cite{games_on_AF_algebras}, the authors prove the following generalisation of Elliott's classification of AF-algebras, in terms of infinitary metric logic.
\begin{theorem}[{\cite[Theorem A]{games_on_AF_algebras}}]
    For all unital AF-algebras $A, B$, and any $\alpha < \omega_1$,
    \[
        (K_0(A), K_0(A)_+, [1_A]_0) \equiv_{\omega \cdot \alpha} (K_0(B), K_0(B)_+, [1_B]_0) \quad \implies \quad  A \equiv_\alpha B.
    \]
\end{theorem}
The proof from \cite{games_on_AF_algebras} combines the intertwining argument of the Elliott classification program with a game-theoretic description of $\equiv_{\alpha}$. 
This naturally leads to the question of whether the same can be done for unital Kirchberg algebras satisfying the UCT, see \cite[Question 1.2]{games_on_AF_algebras}. 
In this paper, we not only answer this question, but also extend the result to all classifiable $C^*$-algebras, and establish the converse, answering in particular \cite[Question 4.5]{games_on_AF_algebras}.
Below, we write $KT_u(\cdot)$ for the functor taking a $C^*$-algebra to its classifying invariant consisting of $K$-theory and traces. We lay this out more explicitly in Section \ref{subsec: invariant}, where we also describe a metric logic language for classifying invariants of $C^*$-algebras. Our main result is the following.
\begin{main}\label{main1: classifiable}
    There exist functions $\theta \colon \omega_1 \rightarrow \omega_1$ and $\theta' \colon \omega_1 \rightarrow \omega_1$ such that for every pair of classifiable $C^*$-algebras $A$ and $B$ and every $\alpha < \omega_1$,
    \begin{align*}
        &KT_u(A) \equiv_{\theta(\alpha)} KT_u(B) \quad &\implies \quad &A \equiv_\alpha B,\\
        &A \equiv_{\theta'(\alpha)} B \quad &\implies \quad &KT_u(A) \equiv_\alpha KT_u(B).
    \end{align*}
        
\end{main}

\subsection{Methods}
In contrast to \cite{games_on_AF_algebras}, rather than going directly through the intertwining arguments of $C^*$-classification, our approach reduces the proof of Theorem \ref{main1: classifiable} to checking that the set of classifiable $C^*$-algebras, the functor $KT_u(\cdot)$, and the relations $\equiv_\alpha$ satisfy certain descriptive `tameness' conditions.
This allows us to handle all classifiable $C^*$-algebras at once, using Theorem \ref{intro thm: classification of C*-algebras}.
More precisely, the following theorem proved in Section \ref{subsection: barwise} is the descriptive set theoretic heart of our method.

\begin{main}\label{main2: DST transfer}
    Let $X, Y$ be analytic\footnote{An analytic set is the image of a Borel set under a Borel map.} subsets of a standard Borel space and let $K \colon X\to Y$ be an analytic map. Assume that $E\subseteq X^2$ and $F\subseteq Y^2$ are Borel equivalence relations, and $(E_\alpha)_{\alpha<\omega_1}$ and $(F_\alpha)_{\alpha<\omega_1}$ are refining uniformly analytic equivalence relations on $X$ and $Y$ respectively, and denote $E_{\infty} = \bigcap_{\alpha <\omega_1} E_\alpha, F_{\infty} = \bigcap_{\alpha <\omega_1} F_\alpha$. Suppose 
    \[ \forall x,x'\in X: \ \begin{cases}
         K(x) \  F_{\infty} \ K(x') \quad &\implies x \ E \ x' ,\\
         x \ E_\infty \ x'\quad &\implies K(x) \ F \ K(x').
    \end{cases}       
    \]
    Then there exist $\alpha,\beta<\omega_1$ such that 
    \[ \forall x,x'\in X: \ \begin{cases}
         K(x) \  F_{\alpha} \ K(x') \quad &\implies x \ E \ x' ,\\
         x \ E_{\beta} \ x'\quad &\implies K(x) \ F \ K(x').
    \end{cases}
    \]
\end{main}
The term `refining' here means that $E_\alpha\supseteq E_\beta$ whenever $\alpha<\beta$, and `uniformly analytic' means the equivalence relations extend to a family of analytic equivalence relations indexed by the set of countable linear orders (see Definition \ref{definition:uniformly_borel_equivalence}).
Our proof is non-constructive, and no bound on $\alpha$ and $\beta$ can be extracted in general. 
For our applications to classification of $C^*$-algebras, we derive the following corollary of Theorem \ref{main2: DST transfer} and the metric Scott's isomorphism theorem from \cite{yaacov2017metric}.  
Here $L_i$ is a metric language and $\Mod_\omega(L_i)$ denotes a suitable encoding of the set of isomorphism classes of separable $L_i$-structures, for $i=1,2$. The set $\Mod_\omega(L_i)$ carries a natural Borel structure, described in Section \ref{subsection: Borel categories}. 
In the formulation below we identify a code for a separable $L_i$-structure $A \in \Mod_\omega(L_i)$ with the structure it encodes.

\begin{maincor} \label{corollary:intro comparison of Scott degrees}
    Let $X\subseteq\Mod_\omega(L_1)$ and $Y\subseteq\Mod_\omega(L_2)$ be analytic subsets and let $K \colon X\to Y$ be an analytic map such that 
    \[ K(A) \cong K(B) \iff A \cong B. \]
    Then there exist maps $\theta\colon\omega_1\to\omega_1, \, \theta'\colon\omega_1\to\omega_1$ such that 
    \begin{equation*}
        \forall A,B\in X:
        \begin{cases}
            K(A) \equiv_{\theta(\alpha)} K(B)&\implies A\equiv_\alpha B, \\
            A\equiv_{\theta'(\alpha)}B&\implies K(A)\equiv_\alpha K(B).
        \end{cases}
    \end{equation*}
\end{maincor}

Rather than applying Theorem \ref{main2: DST transfer} directly to the relations $\equiv_\alpha$ in order to derive the above corollary, we work with a family of equivalence relations, denoted $\equiv_\alpha^{bf}$, coming from so-called \emph{back-and-forth games}, or \emph{Ehrenfeucht-Fra\"iss\'e games} \cite{fraisse1955quelques, ehrenfeucht1957application}. In the discrete setting, given two structures $A$ and $B$ of the same language, the \emph{dynamic Ehrenfeucht-Fra\"iss\'e game} with game clock $\alpha<\omega_1$, denoted $\EFD_\alpha(A,B)$, is played between two players, the \emph{spoiler} and the \emph{duplicator}. Roughly speaking, the {duplicator} wants to prove similarity of $A$ and $B$ and the spoiler wants to prevent this. When the duplicator has a winning strategy in $\EFD_\alpha(A,B)$, we write $A\equiv_\alpha^{bf} B$. The game clock $\alpha$ gives the spoiler some control over the length of the game. A classical result of Karp \cite{KARP2014407} establishes, in the discrete setting, that $\equiv_\alpha^{bf}$ coincides with $\equiv_\alpha$. 
We describe in Section \ref{subsection:metric_games} a version of $\EFD_\alpha(A,B)$ when $A$ and $B$ are structures of a fixed \emph{metric} language, based on the back-and-forth pseudo-distance of \cite{yaacov2017metric}. Several other back-and-forth games for metric structures have been considered in the literature \cite{games_on_AF_algebras, hanson2023approximate, HIRVONEN2022103123}, and a similar treatment of the back-and-forth pseudo-distance in terms of games was recently given in \cite{hirvonen2024ehrenfeuchtfraissegamescontinuousfirstorder}.  
Again writing $A\equiv_\alpha^{bf} B$ if the duplicator has a winning strategy in (the metric version of) $\EFD_\alpha(A,B)$, it is still true that $A\equiv_\alpha B\implies A\equiv_\alpha^{bf} B$ (see Fact \ref{fact: bnf pseudo-metric properties}), but for technical reasons\footnote{ $\equiv_\alpha^{bf}$ has an implicitly fixed additional parameter $\Omega$ that is a \textit{universal weak modulus} as in \cite[Section 5]{yaacov2017metric} and $\equiv_\alpha^{bf}$ coincides with agreement on formulas of quantifier rank $\alpha$ respecting $\Omega$. This is described in more detail in Section \ref{subsection:metric_games}.} arising in the metric setting, $\equiv_\alpha$ and $\equiv_\alpha^{bf}$ do not coincide in general. 
Crucially for the proof of Corollary \ref{corollary:intro comparison of Scott degrees}, we show the following.
\begin{main}\label{theorem: intro uniform Borelness of games}
    Let $L$ be a countable metric language. Then the refining family of equivalence relations $(\equiv_\alpha^{bf})_{\alpha<\omega_1}$ on $\Mod_\omega(L)$
    is uniformly Borel, and in particular uniformly analytic.
\end{main}

To conclude Theorem \ref{main1: classifiable}, we are left with showing that Corollary \ref{corollary:intro comparison of Scott degrees} can be applied to the set of classifiable $C^*$-algebras and the invariant $KT_u(\cdot)$.
The descriptive set theory of the Elliott invariant, and of various assumptions of the unital classification theorem have been previously studied in \cite{farah2013descriptive}. 
In particular, using the results of Farah, Toms, and T\"ornquist, we reduce the proof of Theorem \ref{main1: classifiable} to establishing that the set of $C^*$-algebras satisfying the UCT is an analytic set (with respect to a natural Borel structure on the set of unital separable $C^*$-algebras, see Section \ref{subsec: borel parametrization of C*}).

Rosenberg and Schochet established their universal coefficient theorem (UCT) \cite{RosenbergSchochetUCT}, which computes $KK$-theory\footnote{For $C^*$-algebras $A, B$, $KK(A,B)$ is an abelian group that can be defined in terms of Kasparov $(A,B)$-bimodules. We use a different definition outlined in what follows.} in terms of $K$-theory, for a large class of $C^*$-algebras. 
They show that a $C^*$-algebra satisfies the UCT if and only if it is $KK$-equivalent to an abelian $C^*$-algebra (for our purposes this can be taken as the definition of the UCT). 
This condition plays a central and mysterious role within the classification of $C^*$-algebras; it is unknown whether it is redundant and begs the open question of whether every simple separable nuclear $C^*$-algebra satisfies the UCT. 
One can think of $KK$-equivalence as a loose notion of homotopy equivalence between $C^*$-algebras, and we study the descriptive complexity of the UCT class by using the picture of $KK$ as a `universal functor'. More precisely, from the work \cite{Higson-KKtheory} of Higson, it follows two $C^*$-algebras are $KK$-equivalent if and only if their images under any functor that satisfies certain conditions (that we describe in Section \ref{subsec: borel parametrization of C*}, see Definition \ref{definition: KK-equivalence}) become equivalent. This allows us to present $KK$-equivalence of $C^*$-algebras $A, B$ in an existential way: there exist a finite number of conditions on functors $F$, that imply equivalence of images $F(A) \cong F(B)$.

In order to show that $KK$-equivalence is an analytic relation (which implies analyticity of the UCT), we define a first-order language $L_{KK}$ and an $L_{KK}$-theory $T_{KK}$, whose models correspond to functors satisfying Higson's conditions. 
We prove by a thorough analysis of Higson's conditions that $T_{KK}$ is analytic as a subset of the set of $L_{KK}$-sentences with a suitable Borel space structure. 
This is carried out in Section \ref{section:Borel_properties_of_C_star_alg} and uses a particular construction of the Borel category $\Cstar$ of separable $C^*$-algebras. More precisely, in Section \ref{subsection: Borel categories} we describe a Borel category equivalent to the category of separable models of a fixed metric $L_{\omega_1,\omega}$-sentence with $L$-homomorphisms as morphisms, that we use in the case of $C^*$-algebras (see Proposition \ref{proposition: Borel model of the category of sigma}). Let us note that if one considers only surjective isometries between models as morphisms, such Borel groupoids were already considered in \cite{ElliottEtAl2013, MR3660238} (using the Urysohn space) and in \cite{yaacov2017metric}.
As a corollary of analyticity of $T_{KK}$, we get the following, which is a result of interest in its own right. 
\begin{main}\label{main4: KK-equivalence}
    $KK$-equivalence is an analytic equivalence relation on the category of separable unital $C^*$-algebras. Consequently, the class of unital separable $C^*$-algebras satisfying the UCT is analytic.
\end{main}
Returning to the classification of $C^*$-algebras, the proof of Theorem $\ref{main1: classifiable}$ goes as follows:
\begin{enumerate}
    \item We prove Theorem \ref{main4: KK-equivalence}, showing that the class of unital separable $C^*$-algebras satisfying the UCT is analytic. 
    This is the content of Section \ref{section:Borel_properties_of_C_star_alg} which rests on the formalism of Borel categories of $L$-structures in Sections \ref{subsection: Borel categories} and \ref{subsection: analytic conditions}.
    \item It follows from previous works on the descriptive set theory of $C^*$-algebra invariants (together with Theorem \ref{main4: KK-equivalence}) that the set of classifiable $C^*$-algebras is analytic, and that the map $KT_u(\cdot)$ taking a $C^*$-algebra to its invariant is Borel.
    \item We then apply Corollary \ref{corollary:intro comparison of Scott degrees}, with $X$ being the set of classifiable $C^*$-algebras and the analytic function $K$ being $KT_u(\cdot)$, in order to obtain Theorem \ref{main1: classifiable}. The proof of Corollary \ref{corollary:intro comparison of Scott degrees} uses metric dynamic Ehrenfeucht-Fra\"iss\'e games and Theorem \ref{main2: DST transfer}.
\end{enumerate}

\subsection*{Acknowledgements}
We are indebted to Udi Hrushovski and Stuart White for copious amounts of guidance. 
We thank the participants of the reading group on \cite{games_on_AF_algebras} in Oxford, Michaelmas 2024, in particular Jakub Curda, Shanshan Hua and Austin Shiner. Moreover, we thank Filippo Calderoni, Leo Gitin, and Andrea Vaccaro for some conversations.

\noindent
\textbf{Funding.} JP was supported by the Engineering and Physical Sciences Research Council (EP/X026647/1).
MS was supported by the National Science Foundation under Grant No. DMS-2424441, and by the IAS School of Mathematics. 

For the purpose of Open Access, the authors have applied a CC BY public copyright license to any Author Accepted Manuscript (AAM) version arising from this submission.

\section{Descriptive set theory and infinitary metric logic}\label{subsection: Metric Infinitary Logic}

The model theory of metric structures, dating back to work of Chang and Keisler \cite{chang_keisler_continuous}, was developed extensively in the past two decades, and large parts of (discrete) model theory have been generalized to the metric context \cite{MR2436146, Ben_Yaacov_2010}. Metric model theory has since received growing attention, particularly within the study of operator algebras \cite{farah2021MTofC*, Goldbring2023book, goldbring2022existentially, ioana2024existential, robert2025selfless}. The study of infinitary logics
emerged in the 1960s, with particular focus on the logic $L_{\omega_1, \omega}$, where one allows disjunction and conjunction over countable sets of formulas.\footnote{The subscript $\omega_1$ indicates that only countable disjunctions and conjunctions appear in $L_{\omega_1, \omega}$-formulas, while the subscript $\omega$ indicates that only finite-length blocks of quantifiers appear.} 

We aim to give in Section \ref{sec: infinitary metric logic} a brief introduction to the lesser studied area of infinitary metric logic as treated in \cite{yaacov2017metric}, as well as its connection to descriptive set theory through considering spaces of separable models. We define the notions of a metric language $L$, the corresponding infinitary logic $L_{\omega_1, \omega}$ and the quantifier rank $\qr(\varphi)$ of an $L_{\omega_1, \omega}$-formula, and describe the standard Borel space $\Mod_\omega(L)$ of separable $L$-structures.
We will assume familiarity with some basic concepts in discrete first-order logic. 

In section \ref{subsection:metric_games} we give a presentation of the \textit{back-and-forth pseudo-distance} of Ben Yaacov, Doucha, Nies, and Tsankov \cite{yaacov2017metric}, in terms of a game which we refer to here as the \textit{dynamic metric Ehrenfeucht-Fra\"iss\'e game} $\EFD_{\alpha}$. We show that the set of winning strategies for such games is Borel, and derive in particular that the family of equivalence relations $(\equiv_\alpha^{bf})_{\alpha<\omega_1}$ between separable structures is Borel in a uniform sense, see Theorem \ref{theorem:back-and-forth_is_uniformly_borel}.

We apply in Section \ref{subsection: barwise} a classical fact from descriptive set theory -- namely non-analyticity of the class of countable well orderings -- to derive Theorem \ref{main2: DST transfer}. 

We introduce in Section \ref{subsection: Borel categories} the notion of a Borel category and show that spaces of separable models of countable metric logic theories form Borel categories. The motivation for this is two-fold. Abstractly, we are proposing Borel categories as a setting for studying spaces of models that facilitates both the categorical viewpoint and the descriptive set theoretic viewpoint. Concretely, we make use of this setting in our proof of analyticity of the UCT in Section \ref{section: analycity of KK}, where the encoding of the space of separable C*-algebras as a Borel category plays a central role.

\subsection{Descriptive set theory preliminaries}\label{subsection: DST preliminaries}
We give here a brief summary of some standard definitions and facts from descriptive set theory that will be relevant throughout. For standard references on the subject, see \cite{MR1321597, MR561709}.

Recall that a topological space is called separable if it has a countable dense subset, and it is called Polish if it is separable and completely metrizable. A measurable space is a set together with a $\sigma$-subalgebra of its subsets. For any topological space $X$, denote by $\mathbf{B}({X})$ the smallest collection of subsets of $X$ which contains all open sets and is closed under countable unions and complements. $\mathbf{B}({X})$ forms a $\sigma$-algebra on $X$, called the \textit{Borel algebra} of $X$. Elements of $\mathbf{B}({X})$ are called \textit{Borel subsets} of ${X}$. The measurable space $(X,\mathbf{B}({X}))$ is called the \textit{Borel space} of $X$. We call a measurable space a \textit{standard Borel space} if it is isomorphic to the Borel space associated with a Polish topological space. Note that the union and product of a countable family of Polish spaces is a Polish space. Similarly, the union and product of a countable family of standard Borel spaces is a standard Borel space. A map between standard Borel spaces is called Borel if preimages of Borel sets are Borel.

Kuratowski's theorem says that the only standard Borel spaces up to isomorphism are $\mathbb{R}$, $\mathbb{Z}$, and finite discrete sets. Hence, any two uncountable standard Borel spaces are isomorphic.

A subset of a standard Borel space is called \textit{analytic} if it is the image of a Borel subset by a Borel map.  
Equivalently, analytic subsets can be defined as projections of closed (or Borel) subsets of a product with another standard Borel space. A fact that we will use throughout is that analytic sets are closed under countable unions and intersections, and images and preimages of analytic sets under Borel functions are analytic. For a standard Borel space $X$ and $X_0\subseteq X$, we say that a set $A\subseteq X_0$ is analytic (resp. Borel) relative to $X_0$ if it is the intersection of an analytic (resp. Borel) subset of $X$ with $A$. Since analytic sets are closed under finite intersections, if $X_0$ is analytic then a set is analytic relative to $X_0$ if and only if it is analytic in $X$. 

We call a subset of a standard Borel space coanalytic if its complement is analytic. 
\begin{fact}[Suslin's theorem] \label{fact: suslin's thm}
    A subset of a standard Borel space is Borel if and only if it is analytic and co-analytic.
\end{fact}
As a consequence of Suslin's theorem, a map between standard Borel spaces whose graph is analytic (as a subset of the product) is a Borel function (with a posteriori Borel graph).  

Further in the text, when we work with sets of the form $2^\omega$, $\omega^\omega$, or $\mathbb{R}^{\omega}$, we equip them with their usual Polish topologies and corresponding standard Borel structures. We equip finite and countable powers of these spaces with the respective product topology, which is again a Polish topology.

\subsection{Infinitary metric logic and the space of separable structures}\label{sec: infinitary metric logic}

Given a countable, discrete relational language $L=\{R_i ~\colon i\in I\}$, each countably infinite $L$-structure is isomorphic to an $L$-structure $A=(\omega;(R_i^A)_{i\in I})$ on $\omega$. This structure $A$ can be viewed as an element of the set $\Mod_\omega(L):=\prod_{i\in I}2^{\omega^{n_i}}$ where $n_i$ denotes the arity of $R_i$. As a countable product of finite powers of the Cantor space, $\Mod_\omega(L)$ carries a natural Polish topology, which coincides with the topology generated by open sets of the form $\Mod_\omega(\varphi(a)):=\{A \colon A\models \varphi(a)\}$ where $\varphi$ is a quantifier-free formula. This leads to a fruitful interplay between (infinitary) first-order logic and descriptive set theory that has been studied extensively.  
Sets of the form $\Mod_\omega(\sigma)$ where $\sigma$ is an $L_{\omega_1, \omega}$-formula are Borel in $\Mod_\omega(L)$, as can be seen by induction. A theorem of Lopez-Escobar \cite{Lopez1965} establishes that conversely, if $X\subseteq \Mod_\omega(L)$ is Borel and invariant under the action of $S_\infty$, then $X=\Mod_{\omega}(\sigma)$ for some $L_{\omega_1, \omega}$-sentence $\sigma$ (see also \cite[Chapter II.16]{MR1321597}). Scott's isomorphism theorem for countable structures produces for each $A\in\Mod_\omega(L)$ an $L_{\omega_1, \omega}$-sentence $\sigma_A$ called the \textit{Scott sentence} of $A$ such that the isomorphism class of $A$ in $\Mod_\omega(L)$ coincides with $\Mod(\sigma_A)$.

In view of our main application in Section \ref{section: Application}, we are interested in the more general case of a continuous logic language $L$, and the study of spaces of separable $L$-structures. Roughly speaking, continuous logic extends discrete first-order logic by allowing relations -- and therefore all formulas -- to take real values and allowing as connectives arbitrary continuous functions. For a detailed treatment of continuous logic, we refer the reader to \cite{MR2436146, MR4654490}.

A key technical subtlety that arises in moving to such a logic is that in order to obtain a compactness theorem, one needs to ensure that each formula takes values only in some bounded interval, uniformly in each structure. 
We use here the framework of metric logic, as developed in \cite{MR2436146}, where every language $L$ contains a distinguished binary predicate symbol $d$ to be interpreted as a metric, and an $L$-structure is a complete metric space, with each relation symbol $R$ taking values in a compact interval $I_R\subseteq\mathbb R$, determined by $L$.
In particular, the language $L$ requires each structure to have diameter bounded by some fixed $D_L>0$. It is natural to insist functions and relations be continuous on a given $L$-structure with respect to the metric. 
To ensure that continuity passes through ultraproducts, one demands each function and relation respect some \textit{modulus of uniform continuity}. Given $\Delta \colon [0,\infty)^n\to[0,\infty)$ and a metric space $X$, we define for $x,x'\in X^n$,
\[d_\Delta(x,x'):=\Delta(d_X(x_1,x'_1),\dots,d_X(x_n,x_n')).\]
If $\Delta$ is continuous, non-decreasing and subadditive and $\Delta(0)=0$, then $d_\Delta$ defines a pseudo-metric.
In this case we say that $\Delta$ is a \textit{modulus of continuity} and a map $f \colon X^n\to Y$ \textit{respects} $\Delta$ if $d_Y(f(x),f(x'))\leq d_\Delta(x,x')$ for any $x,x'\in X^n$.

\begin{definition}\label{definition: metric language}
    A metric language $L$ consists of 
    \begin{enumerate}
        \item[(i)] A set of sorts $S$;
        \item[(ii)] For each $s\in S$ a binary relation symbol $d_s$ and a diameter $D_s>0$;
        \item[(iii)] Relation symbols $R$, together with an arity $\overline{s}\in S^n$, a compact interval $I_R\subseteq\mathbb R$ and a modulus of continuity $\Delta_R\colon [0,\infty)^n\to[0,\infty)$;
        \item[(iv)] Function symbols $f$, together with an arity $\overline{s}\in S^{n+1}$ and a modulus of continuity $\Delta_f \colon [0,\infty)^n\to[0,\infty)$;
        \item[(v)] Constant symbols $c$, together with a sort $s\in S$. 
    \end{enumerate}
\end{definition}

Let $L$ be a metric language and assume for notational simplicity that $L$ is one-sorted, that is $|S|=1$.
An $L$-structure ${A}$ consists of a complete metric space $(A,d_A)$ with diameter $\leq D_L$, together with
\begin{itemize}
    \item for each $n$-ary relation symbol $R\in L$ a function $R^A\colon A^{n}\to I_R$ that respects $\Delta_R$;
    \item for each $n$-ary function symbol $f\in L$, a map $f^A\colon A^{n}\to A$ that respects $\Delta_f$;
    \item for each constant symbol $c\in L$, an element $c^A\in A$.
\end{itemize}

${L}$-terms are defined inductively from constant and function symbols as in discrete logic, and each $L$-term $\tau$ carries a modulus of continuity $\Delta_\tau$, obtained inductively from those on function symbols. Atomic formulas are of the form $R(\tau_1,\dots,\tau_{n})$ where $\tau_i$ are terms and $R$ is a relation symbol. ${L}$-formulas are built inductively from atomic formulas, with
arbitrary uniformly continuous maps $u\colon \mathbb R^n\to\mathbb R$ as connectives, and quantifiers $\inf_x$ and $\sup_x$ (in place of $\exists x$ and $\forall x$). 

We move now to infinitary metric logic following \cite{Ben_Yaacov_2009}.
A technical hurdle of considering infinitary formulas in metric languages is that when taking infima and suprema (in place of the discrete infinitary connectives $\bigvee$ and $\bigwedge$) over countable families of formulas, the resulting formula may fail to be continuous or have unbounded range. This motivates the restriction on infinitary conjunctions and disjunctions in the fourth point of Definition \ref{def: Metric Infinitary Formulas} below. 

Fix a metric language $L$. Note that the intervals $I_R$ and moduli of continuity $\Delta_R$ respectively $\Delta_\tau$ on relation symbols respectively terms in $L$ can be extended inductively in a canonical way to compact intervals $I_\varphi$ and moduli of uniform continuity $\Delta_\varphi$, when $\varphi$ is an infinitary formula as in the below definition. We write $\varphi\in L_{\omega_1, \omega}$ when $\varphi$ is an $L_{\omega_1,\omega}$-formula. Recall that $\omega_1$ denotes the smallest uncountable ordinal.

\begin{definition}\label{def: Metric Infinitary Formulas}
    The set of $L_{\omega_1, \omega}$-formulas $\varphi$, as well as their \textit{quantifier rank} $\qr(\varphi)$ is defined as follows.
    \begin{enumerate}
        \item[(i)] If $\varphi$ is a quantifier-free $L$-formula then $\varphi\in L_{\omega_1, \omega}$ with $\qr(\varphi)=0$;
        \item[(ii)] If $\varphi_1,\dots,\varphi_n\in L_{\omega_1, \omega}$ and $u \colon \mathbb R^n\to\mathbb R$ is uniformly continuous then $u(\varphi_1,\dots,\varphi_n)\in L_{\omega_1, \omega}$ with 
        \[\qr\big(u(\varphi_1,\dots,\varphi_n)\big)=\max\big\{\qr(\varphi_1),\dots,\qr(\varphi_n)\big\};\]
        \item[(iii)] If $\varphi\in L_{\omega_1, \omega}$ and $x$ is a variable then $\sup_x\varphi\in L_{\omega_1, \omega}$ and $\inf_x\varphi\in L_{\omega_1, \omega}$ with 
        \[\qr({\sup}_x\varphi)=\qr({\inf}_x\varphi)=\qr(\varphi)+1;\]
        \item[(iv)] Given variables $x_1,\dots,x_n$, a bounded interval $I\subseteq\mathbb R$ and a modulus of continuity $\Delta \colon [0,\infty)^n\to[0,\infty)$, if $(\varphi_i)_{i<\omega}$ is such that for each $i<\omega$, $\varphi_i$ is an $L_{\omega_1, \omega}$-formula with free variables $x_1,\dots,x_n$, taking values in $I$ and respecting $\Delta$ (i.e., $\Delta_{\varphi_i} \leq \Delta$), then $\bigwedge_i\varphi_i\in L_{\omega_1,\omega}$ and $\bigvee_i\varphi_i\in L_{\omega_1, \omega}$ with 
        \[\qr\Big(\bigvee_i\varphi_i\Big)=\qr\Big(\bigwedge_i\varphi_i\Big)=\sup_i\qr(\varphi).\]
    \end{enumerate}
\end{definition}

Note that by definition, $\qr(\varphi)$ is a countable ordinal (written $\qr(\varphi)<\omega_1$) for any $\varphi\in L_{\omega_1, \omega}$. 
The first-order $L$-formulas are those defined only by $(i)-(iii)$ in the above definition. A sentence is a formula without free variables. A (first-order) theory $T$ is a set $T=\{ \varphi_i \colon i \in I \}$ for some (first-order) sentences $\{\varphi_i\}_{i \in I}$. For an $L_{\omega_1, \omega}$-sentence $\sigma$ and an $L$-structure $A$, we write $\sigma^A$ for the value of $\sigma$ interpreted in $A$. We say that $A$ satisfies $\sigma$ (written $A \models \sigma$), if $\sigma^A = 0$. Similarly, for a theory $T$, we write $A \models T$ if $A$ satisfies all sentences from $T$. In that case we call $A$ a model of $T$. For $L$-structures $A$ and $B$, we write $A\cong B$ and say that $A$ and $B$ are \textit{isomorphic} if there exists a bijection $f:A\to B$ with $\varphi^A(a)=\varphi^B(f(a))$ for each $L$-formula $\varphi(x)$ and tuple $a$ in $A$ of appropriate arity. The following gives a family of weaker equivalence relations between $L$-structures.

\begin{definition}\label{definition:alpha_elementary_equivalence}
    For $L$-structures $A,B$ and $\alpha<\omega_1$ we write $A\equiv_\alpha B$ if $\sigma^A=\sigma^B$ for every $L_{\omega_1, \omega}$-sentence $\sigma$ with $\qr(\sigma)\leq\alpha$.
\end{definition} 

We define the language and theory of $C^*$-algebras and of order unit spaces below for use in future applications in Section \ref{section: Application}.

\begin{example}\label{example:language of C*}
    Consider the multi-sorted countable language $L_{C^*}$ for $C^*$-algebras, containing for each $n\in\mathbb{N}$,
    \begin{itemize}
        \item a sort $B_n$ to be interpreted as the ball of radius $n$ and a predicate $d_n \colon B_n\to[0,n]$ to be interpreted as the metric,
        \item function symbols 
        \begin{itemize}
            \item $+_n \colon B_n\times B_n\to B_{2n}$ to be interpreted as addition, 
            \item ${^*}_n \colon B_n\to B_n$ to be interpreted as involution,
            \item $\cdot_n \colon B_n\times B_n\to B_{n^2}$ to be interpreted as multiplication,
            \item $(\cdot\lambda)_n \colon B_n\to B_m$ for $\lambda\in\mathbb Q(i)$ and $m={\lceil |\lambda| n \rceil}$, to be interpreted as scalar multiplication,
            \item $I_{mn} \colon B_m\to B_n$ for $m<n$, to be interpreted as containment,
        \end{itemize}
        \item a constant symbol $0$ of sort $B_1$, to be interpreted as zero.
    \end{itemize}  
    Furthermore, we write $T_{C^*}$ for the theory of $C^*$-algebras as presented in Example~2.2.1 from \cite{farah2021MTofC*}, except we only use $\mathbb{Q}(i)$ scalars to ensure $T_{C^*}$ is countable (see loc.cit. for a natural choice of moduli of continuity and some discussion). We denote by $T_{C^*}'$ the theory of unital $C^*$-algebras in the language $L_{C^*}$ expanded by a constant symbol $1 \in B_1$.
\end{example}

\begin{example} \label{example:language Lous}
    Consider the multi-sorted language $L_{ous}$ (``$ous$'' stands for order unit spaces) containing for each $n\in\mathbb N$,
    \begin{itemize}
        \item a sort $B_n$ with metric predicates $d_n \colon B_n\to[0,n]$,
        \item function symbols 
        \begin{itemize}
            \item $+_n \colon B_n\times B_n\to B_{2n}$, 
            \item $(\cdot\lambda)_n \colon B_n\to B_m$ for $m={\lceil |\lambda| n \rceil}$ and $\lambda\in\mathbb Q$,
            \item $I_{mn} \colon B_m\to B_n$ for $m<n$,
        \end{itemize}
        \item relation symbols $P_n \colon B_n \to [0,n]$,
        \item and constant symbols $0\in B_1$ and $1\in B_1$.
        The moduli of continuity of addition and multiplication symbols are $\Delta_{+_n}(\delta,\delta')=\delta+\delta'$, $\Delta_{\cdot\lambda}(\delta)=|\lambda|\delta$. The moduli of continuity of $I_{mn}$'s and $P_n$'s are identity.
    \end{itemize}
    An order unit space $A$ is an Archimedean real ordered vector space with a distinguished element $\mathbf{1}$ called its order unit, such that for every $a \in A$, $\{k a \colon k \in \mathbb{N} \}$ is unbounded, and there is a natural number $n$ such that $a \leq n \mathbf{1}$. In that case one defines a norm on $A$ by
    \[ \|a \| := \inf \{r \colon -r\mathbf{1} \leq a \leq r \mathbf{1}  \}. \]
    Note that each complete order unit space can be seen as an $L_{ous}$-structure, by interpreting $B_n$ as the ball of radius $n$ in $A$, $d_n$ as the metric $d_n(x,y) = \|x-y\|$, and $P_n$ as the distance to the positive cone.
\end{example}

We will be interested throughout this paper in separable metric structures. Let ${L}=(d;(R_i)_i;(f_j)_j;c_k)$ be a one-sorted countable metric language. Since $L$ is countable, any separable $L$-structure $A$ contains a countable dense subset $A_0\subseteq A$, containing all constants $c_k^A$ and closed under all functions $f_j^A$. Such $A_0$ is sometimes called a \emph{pre-substructure}, and a \emph{substructure} if additionally, $A_0$ is closed as a metric space. When $A_0$ is a dense pre-substructure, $A$ is uniquely determined by $A_0$.

After fixing a countable dense pre-substructure $A_0=\{a_i ~\colon i<\omega\} \subseteq A$, we can encode $A$ as an element of a suitable Polish space. Equip 
\begin{equation} \label{Borel L-structure}
    X_L:=\mathbb R^{\omega^2}\times\prod_i\mathbb R^{\omega^{n_i}}\times\prod_j\omega^{\omega^{m_j}}\times\prod_k\omega 
\end{equation}
with the standard Borel structure induced by the product topology of the respective Polish topologies on $\mathbb R^{\omega^{n_i}}$, $\omega^{\omega^{m_j}}$ and $\omega$. Then we can assign to $A$ an element $A\in X_L$ (by abuse of notation) in the canonical way, setting
\begin{align*}
    A(i):=\begin{cases}
    \omega^{n_i}\to\mathbb R\\
    (j_1,\dots,j_{n_i})\mapsto R_i^A(a_{j_1},\dots,a_{j_{n_i}})
\end{cases}
\end{align*}
and encoding function symbols, constant symbols and the distinguished predicate $d$ in an analogous way.

We define $\Mod_\omega(L)\subseteq X_L$ to be the range of this assignment. We assume in the definition below that $L$ is a countable one-sorted metric language as in Definition \ref{definition: metric language}. Note that the construction can be generalized straightforwardly to a language with countably many sorts. 

\begin{definition}\label{definition: Mod_omega(L)}
    $\Mod_\omega(L)$ is the set of those $(d,(R_i)_{i},(f_j),(c_k))\in \mathbb R^{\omega^2}\times \prod_{i}\mathbb R^{\omega^{n_i}}\times \prod_j\omega^{\omega^{m_j}}\times\prod_k\omega=:X_L$ for which
\begin{itemize}
    \item $d \colon \omega^2\to\mathbb R$ is a metric of diameter $<D$;
    \item $R_i \colon \omega^{n_i}\to\mathbb R$ takes values only in $I_{R_i}$ and respects $\Delta_{R_i}$.
    \item $f_j \colon \omega^{m_j}\to\omega$ respects $\Delta_{f_j}$.
\end{itemize}
\end{definition}

For $A\in\Mod_\omega(L)$, denote the unique separable $L$-structure encoded by $A$ as $\widehat{A}$, that is $\widehat{A}$ is the completion of $A$ with respect to $d$, with functions and relations extended continuously.

The following lemma follows from checking that the properties listed in Definition \ref{definition: Mod_omega(L)} are Borel, and an induction argument.

\begin{lemma}\label{lem: set of separable models is Borel}
    For an $L_{\omega_1, \omega}$-sentence $\sigma$, the set \[ \Mod_\omega(\sigma):=\{A\in\Mod_\omega(L) ~\colon \widehat{A}\models \sigma\} \] is Borel in $X_L$. In particular, $\Mod_\omega(L)$ is Borel.
\end{lemma}
By the metric generalization of Scott's isomorphism theorem proved in \cite[Theorem 1.2]{yaacov2017metric}, the isomorphism class in $\Mod_\omega(L)$ of any separable $L$-structure $A$ is of the form $\Mod_\omega(\sigma_A)$ for some $\sigma_A\in L_{\omega_1, \omega}$. As a corollary, one obtains the following.

\begin{fact}[{\cite[Theorem 1.2]{yaacov2017metric}}] \label{theorem:metric_scott_iso}
    For a countable metric language $L$ and separable $L$-structures $A$ and $B$,
    \[A\cong B\iff \forall\alpha<\omega_1 ~\colon A\equiv_\alpha B.\]
\end{fact}

\subsection{Metric dynamic Ehrenfeucht-Fra\"iss\'e games}\label{subsection:metric_games}

An illustrative way of characterizing the expressive power of an (infinitary) logic comes from so-called \textit{Ehrenfeucht-Fra\"iss\'e games}. Going back to Hausdorff and what is now a standard back-and-forth argument for Cantor's isomorphism theorem \cite[Section 4.7]{hausdorff_grundzuege}, these games were first considered explicitly by Ehrenfeucht \cite{ehrenfeucht1957application}, based on the work of Fra\"iss\'e \cite{fraisse1955quelques}. 

Given structures $A$ and $B$ of some discrete, relational language $L$ and $n<\omega$, the game $\EF_n(A,B)$ is played between two players -- Player I (the \textit{challenger}) and Player II (the \textit{duplicator}). The challenger wants to prove that $A$ and $B$ are different (and that the considered logic can detect this difference), while the duplicator wants to show that they are similar. In round $k$, Player I chooses an element of $A$ or $B$, and Player II responds by choosing an element of the other structure. The played elements $a_k\in A$ and $b_k\in B$ are recorded, and after $n$ rounds, Player II wins if the resulting sequences $a\in A^n$, $b\in B^n$ satisfy the same quantifier-free $L$-formulas. One can check inductively that if $L$ is finite, II has a winning strategy in $\EF_n(A,B)$ precisely when $A\equiv_n B$, that is, the structures agree on sentences up to quantifier rank $n$. 

Denoting by $\EF_\omega(A,B)$ the analogous game with $\omega$-many rounds, II has a winning strategy in $\EF_\omega(A,B)$ if and only if the structures are indistinguishable to the infinitary logic $L_{\omega_1, \omega}$\footnote{In fact, this game corresponds to $L_{\infty,\omega}$, but we are interested only in countable languages where $L_{\infty,\omega}$ and $L_{\omega_1,\omega}$ always coincide.}, i.e., $A\equiv_{\omega_1}B$. To capture the finer notion $\equiv_\alpha$ of elementary equivalence up to formulas of quantifier-rank $\alpha$ for some $\alpha<\omega_1$, one considers a version of the game with dynamic length. The \textit{dynamic Ehrenfeucht-Fra\"iss\'e game} $\EFD_\alpha(A,B)$ is played like $\EF_n(A,B)$, with the following modification. There is a `game-clock', call it $\beta$, that is set to $\alpha$ before the game starts. In each round, Player I, in addition to playing a usual move, sets the game-clock to some ordinal $\beta'<\beta$. The game ends when $\beta=0$. Note that if $\alpha\geq\omega$, a play of the game can be arbitrarily long, but will always be finite since $\alpha$ is well-ordered. The advantage granted to Player I in the game clock amounts exactly to the expressive power gained in increasing quantifier rank: II has a winning strategy in $\EFD_\alpha(A,B)$ if and only if $A\equiv_\alpha B$. This is due to Karp \cite{KARP2014407}.\\

We define in this section a generalization of the dynamic Ehrenfeucht-Fra\"iss\'e game for metric languages based on the back-and-forth pseudo-distance of Ben Yaacov, Doucha, Nies and Tsankov \cite{yaacov2017metric}, and show in Theorem \ref{theorem:back-and-forth_is_uniformly_borel} that the resulting refining sequence of equivalence relations $(\equiv^{bf}_\alpha)_{\alpha<\omega_1}$ is uniformly Borel in the sense of Defintion \ref{definition:uniformly_borel_equivalence}. 
The na\"ive approach to defining metric EF-games would be to play exactly as in the discrete case, with II winning if $\sup_\varphi|\varphi^A(a)-\varphi^B(b)|=0$ where $\varphi$ ranges over all finitary, quantifier-free formulas. This game is too strong, however. When we consider separable structures $A$ and $B$, we want to restrict the legal moves to some previously fixed dense countable substructures of $A$ and $B$, without changing the outcome of the game. The above described game is sensitive to this restriction: Take $A\in\Mod_\omega(L_{C^*})$ to be a separable $C^*$-algebra, with a countable dense pre-substructure $A_0 \subseteq A$. Find $r\in\mathbb R\setminus\mathbb Q$ such that $||a||\neq r$ for all $a\in A_0$. Setting $A_0'$ to be the (still countable) pre-substructure generated by $A_0\cup\{r\}$, we see that Player I can win in one round in a game between $A_0$ and $A_0'$ -- both of which have completion $A$ -- by playing $r\in A_0'$. We remedy this by loosening the winning requirement for Player II such that values of quantifier-free formulas may differ up to some $\varepsilon>0$ on the played tuples. Note however that this only makes sense if we impose some restriction on the formulas we consider. Indeed, if $a\in A^n$, $b\in B^n$ are such that $|\varphi^A(a)-\varphi^B(b)|>0$ for some quantifier-free $\varphi$, then 
\[\sup_\psi|\psi^A(a)-\psi^B(b)|\geq \sup_{n\in\mathbb N}|(n\varphi)^A(a)-(n\varphi)^B(b)|=\infty\]
where $\psi$ ranges over all quantifier-free formulas.
Following \cite{yaacov2017metric}, this is dealt with by restricting only to formulas $\psi$ which respect a fixed, so-called \textit{weak modulus} $\Omega \colon [0,\infty)^{\mathbb N}\to[0,\infty]$ in the sense of \cite[Definition 2.2]{yaacov2017metric}. That is, the truncations
\[\Omega|_n \colon \begin{cases}
    [0,\infty)^n\to[0,\infty)\\
    (\delta_1,\dots,\delta_n)\mapsto\Omega(\delta_1,\dots,\delta_n,0,0\dots)
\end{cases}\]
are continuous, non-decreasing and sub-additive, and $\Omega$ is determined by the truncations via $\Omega(\delta)=\sup_n\Omega|_n(\delta|_n)$. For a metric language $L$, the set of $L_{\omega_1, \omega}$-$\Omega$-formulas is defined to be the closure of the set of all quantifier-free finitary $L$-formulas under $1$-Lipschitz connectives $u \colon \mathbb R^n\to\mathbb R$, quantification over the free variable of maximal index\footnote{This is to ensure that all $n$-ary $\Omega$-formulas will respect $\Omega|_n$.}, and conjunctions and disjunctions over uniformly bounded sequences of formulas. Note that in contrast to Definition \ref{def: Metric Infinitary Formulas} (iv), there is no additional requirement for equicontinuity when taking infinite disjunctions and conjunctions, since any $n$-ary $\Omega$-formula will respect $\Omega|_n$. For $L$-structures $A$ and $B$, write $A\equiv_\alpha^\Omega B$ if the structures agree on $\Omega$-formulas up to quantifier-rank $\alpha$.

The notion of a dynamic metric Ehrenfeucht-Fra\"iss\'e game defined below is already implicit in the following definition. 
As usual, we denote the concatenation of tuples $a$ and $b$ simply by $ab$.

\begin{definition}[{\cite[Definition 3.1]{yaacov2017metric}}] \label{definition:back-and-forth-distance}
    Let $\Omega \colon [0,\infty)^{\mathbb N}\to[0,\infty]$ be a weak modulus, $n<\omega$ and $\alpha<\omega_1$. For $L$-structures $A$, $B$ and $a\in A^n$, $b\in B^n$, the {back-and-forth pseudo-distance} $r_{\alpha,n}^\Omega(Aa,Bb)$ is defined inductively as
    \begin{itemize}
        \item $r_{0,n}^{\Omega}(Aa,Bb):=\sup_\varphi|\varphi^{A}(a)-\varphi^{B}(b)|$
        where $\varphi$ ranges over finitary quantifier-free $\Omega$-formulas;
        \item $r_{\alpha+1,n}^{\Omega}(Aa,Bb):=\sup_{c\in A,d\in B}\inf_{c'\in A,d'\in B}(\max(r^\Omega_{\alpha,n+1}(Aac,Bbd'),r^\Omega_{\alpha,n+1}(Aac',Bbd)));$
        \item For $\gamma$ a limit, 
        \[r^\Omega_{\gamma,n}(Aa,Bb):=\sup_{\beta<\gamma}r_{\beta,n}^\Omega(Aa,Bb).\]
    \end{itemize}
    Write $A\equiv_\alpha^{bf,\Omega}B$ whenever $r_{\alpha,0}^\Omega(A,B)=0$.
\end{definition}

We restrict attention to separable $L$-structures, which we view as elements of the Borel space $\Mod_\omega(L)$, as discussed above Definition \ref{definition: Mod_omega(L)}. 
In particular, if $A,B \in \Mod_{\omega}(L)$ and $a,b\in\omega^n$, we write $r_{\alpha,n}^\Omega(Aa,Bb)$ for $r_{\alpha,n}^\Omega(\widehat{A}a,\widehat{B}b)$.
By continuity of $r_{\alpha,n}^\Omega$ \cite[Lemma 3.2]{yaacov2017metric}, it is sufficient in the separable case to restrict the range of supremum and infimum in the successor step to a distinguished countable dense set.
An inductive argument shows the following. 

\begin{fact}[{\cite[Proposition 4.3]{yaacov2017metric}}] \label{fact: pseudo bnf distance is Borel}
    $r_{\alpha,n}^{\Omega} \colon {(\Mod_\omega(L)\times\omega^n)}^2\to\mathbb R^{\geq 0}$ is Borel.
\end{fact}

The following fact verifies that $r_\alpha^\Omega$ measures similarity up to $\Omega$-formulas of quantifier-rank $\alpha$.

\begin{fact}[{\cite[Theorem 3.5]{yaacov2017metric}}] \label{fact: bnf pseudo-metric properties}
    $r_{\alpha,n}^{\Omega}(Aa,Bb)=\sup_\varphi|\varphi^{{A}}(a)-\varphi^{B}(b)|$ where $\varphi$ ranges over $n$-ary $\Omega$-formulas of quantifier rank $\leq \alpha$. In particular $A\equiv_\alpha^\Omega B$ if and only if $A\equiv_\alpha^{bf,\Omega}B$.
\end{fact}

While back-and-forth pseudo-distance only considers formulas respecting some fixed weak modulus $\Omega$, this restriction is overcome in \cite[Chapter 5]{yaacov2017metric}. Namely, the authors prove existence of a so-called \textit{universal modulus} $\Omega$ such that for any $L_{\omega_1, \omega}$-sentence $\sigma$, there is an equivalent $L_{\omega_1,\omega}$-$\Omega$-sentence $\sigma'$. The following fact summarizes this.

\begin{fact}[{\cite[Section 5]{yaacov2017metric}}]\label{fact: universal modlus}
    There exists a weak modulus $\Omega$ such that for any $L$-structures $A$ and $B$,
    \[\big(\forall \alpha<\omega_1:~A\equiv_\alpha^{bf,\Omega} B\big) \iff A\cong B.\]
\end{fact}

Note that this universal modulus depends on the language $L$. We fix from here onward $\Omega$ to be a modulus as in Fact \ref{fact: universal modlus} and omit it in our notation. In particular, we write $r_{\alpha,n}$ for the back-and-forth pseudo-distance with respect to the fixed universal modulus, and for $A,B\in\Mod_\omega(L)$ we write $A\equiv_\alpha^{bf}B$ if $r_{\alpha,0}(A,B)=0$. 

We now turn to characterizing the pseudo-distance $r_{n,\alpha}$ in terms of a game.
Denote by $\LO\subseteq\mathcal{P}(\omega)\times\mathcal{P}(\omega^2)$ the set of linear orders on subsets of $\omega$.
For $L$ a metric language, $L$-structures $A$ and $B$ and $(\lambda,\prec)\in\LO$ the \emph{metric dynamic Ehrenfeucht-Fra\"iss\'e Game} $\EFD_{\prec,\varepsilon}(A,B)$ is played between Players I and II in the following way. The \textit{game clock} is set first to $\beta_0:=\infty$. For $k\geq 1$, the $k$-th round is played as follows.
\begin{enumerate}
    \item Player I plays either a pair $(\beta_k,a_k)$ or a pair $(\beta_k,b_k)$, where $\beta_k\in\lambda$ with $\beta_k\prec\beta_{k-1}$, and $a_k\in A$ or $b_k\in B$.
    \item Player II responds by playing $b_k\in B$ if I played an element $a_k\in A$, or by playing $a_k\in A$ if I played an element $b_k\in B$.
\end{enumerate}
We call the tuple $(\beta_k,(a_1,\dots,a_k),(b_1,\dots,b_k))$ the \textit{position at $k$}. The game ends after round $k$ if $\beta_k$ is the minimum of $(\lambda,\prec)$ (if a minimum exists). 
In this case, Player II wins if and only if $r_{0,k}(Aa_1\dots a_k,Bb_1\dots b_k)<\varepsilon$. Note that if $\prec$ is not well-ordered, the game may go on for countably infinitely many rounds, after which infinite tuples $a\in A^\omega$ and $b\in B^\omega$ will have been played and the game ends. In this case, Player II wins if and only if $r_{0,k}(Aa_1\dots a_k,Bb_1\dots b_k)<\varepsilon$ for all $k<\omega$.  

While we have defined $\EFD_{\prec,\varepsilon}(A,B)$ for arbitrary $L$-structures, we will be interested only in separable ones, viewed as elements of the space $\Mod_\omega(L)$ defined in Section \ref{subsection: Metric Infinitary Logic}. In a game on separable structures $A,B\in\Mod_\omega(L)$, we require all moves $a_k$ and $b_k$ to be elements of the fixed countable dense subset of $A$ and $B$ respectively, which we identify with $a_k,b_k\in\omega$. We will define what a \textit{winning strategy} for Player II is in terms of separable structures, but this can straight-forwardly be generalized to arbitrary $L$-structures. 
Intuitively, a strategy for II gives for each potential position of the game at least one legal move that II can respond with. Such a strategy is winning for II if II is guaranteed to win as long as they play by this strategy. 

\begin{definition}\label{definition: metric EFD winning strategy for II}
    Let $A,B\in \Mod_\omega(L)$, $(\lambda,\prec)\in\LO$ and $\varepsilon>0$. A winning strategy for $\mathrm{II}$ in $\EFD_{\prec,\varepsilon}(A,B)$ is a sequence $(S_n)_{n\geq 1}$ with $S_n\subseteq \omega^n\times\omega^n\times\omega^n$ such that 
    \begin{itemize}
        \item $\forall \beta\in\lambda\ \forall c\in\omega \ \exists d,d'\in\omega ~\colon (\beta,c,d),(\beta,d',c)\in S_1$,
    \end{itemize}
    and for all $n<\omega$ and $(\beta,a,b)\in S_n$,
    \begin{itemize}
        \item $\forall \beta_{n+1}\prec\beta_n\ \forall c\in \omega\ \exists d,d'\in\omega ~\colon (\beta\beta_{n+1},ac,bd),(\beta\beta_{n+1},ad',bc)\in S_{n+1}$;
        \item $r_{0,n}(Aa,Bb)<\varepsilon$.
    \end{itemize}
    We write $\mathrm{II} \uparrow\EFD_{\prec,\varepsilon}(A,B)$ if there exists a winning strategy for Player $\mathrm{II}$, and we denote the set of winning strategies for $\mathrm{II}$ by $\mathcal{S}_{\prec,\varepsilon}^{A,B}(\mathrm{II})\subseteq \prod_n\mathcal{P}(\omega^n\times\omega^n\times\omega^n)$.
\end{definition}

Analogously, we define the notion of a winning strategy for Player I. For $(S_n)_{n\in\mathbb N}\in \prod_{n\in\mathbb N}\omega^n\times\omega^n\times\omega^n$ and $(\beta,a,b)\in S_n$, we write $S_{n+1}(\beta,a,b):=\{(\beta\beta_{n+1},ac,bd)\in S_{n+1}\}$.

\begin{definition}\label{definition: metric EFD winning strategy for I}
    Let $A,B\in \Mod_\omega(L)$, $(\lambda,\prec)\in \LO$ and $\varepsilon>0$. A winning strategy for $\mathrm{I}$ in $\EFD_{\prec,\varepsilon}(A,B)$ is a sequence $(S_n)_{n\in\mathbb N}$ with $S_n\subseteq \omega^n\times\omega^n\times\omega^n$ such that 
    \begin{itemize}
        \item $\exists \beta\in\lambda \ \exists c\in\omega ~\colon \big(\forall d\in\omega ~\colon (\beta,c,d)\in S_1\big)\ \vee \ \big(\forall d\in\omega~\colon (\beta,d,c)\in S_1\big)$,
    \end{itemize}
    and for any $(\beta,a,b)\in S_n$,
    \begin{itemize}
        \item if $\beta_n\neq \min(\lambda,\prec)$ then $\exists\beta_{n+1}\prec\beta_n\ \exists c\in \omega $ s.t. $ S_{n+1}(\beta,a,b)=\{(\beta\beta_{n+1},ac,bd)\colon d\in\omega\}$ or $S_{n+1}(\beta,a,b)=\{(\beta\beta_{n+1},ad',bc)\colon d'\in\omega\};$
        \item if $\beta_n=\min(\lambda,\prec)$ then $r_{0,n}(Aa,Bb)\geq\epsilon$;
        \item $\exists k<\omega\ \exists (\beta\beta',ac,bd)\in S_{n+k}~\colon \ r_{0,n+k}(Aac,Bbd)\geq\varepsilon$.
    \end{itemize}
    We write $\mathrm{I}\uparrow\EFD_{\prec,\varepsilon}(A,B)$ if there exists a winning strategy for Player $\mathrm{I}$, and denote the set of winning strategies for $\mathrm{I}$ by $\mathcal{S}_{\prec,\varepsilon}^{A,B}(\mathrm{I})\subseteq \prod_n\mathcal{P}(\omega^n\times\omega^n\times\omega^n)$.
\end{definition}

For $a\in \omega^n$ and $A\in\Mod_\omega(L)$, denote by $Aa$ the expansion of $A$ by an $n$-ary constant symbol interpreted as $a$. This is a separable structure in the language $L\cup\{c_1,\dots,c_n\}$. Note that the game $\EFD_{\prec,\varepsilon}(Aa,Bb)$ is equivalent to the game $\EFD_{n+\prec,\varepsilon}(A,B)$, with $a$ and $b$ as the forced first $n$ moves, in the sense that a winning strategy for II in either one of these games translates to one in the other.

The metric dynamic Ehrenfeucht-Fra\"iss\'e game is defined precisely in a way that models the back-and-forth pseudo-metric. This is formalized in the following.

\begin{proposition}\label{proposition: Games correspond to bnf pseudo-distance}
    Let $A,B\in\Mod_\omega(L)$, $a,b\in\omega^n$, $\alpha<\omega_1$ and $\varepsilon>0$. Then $\mathrm{II} \uparrow\EFD_{\alpha,\varepsilon}(Aa,Bb)$ if and only $r_{n,\alpha}(Aa,Bb)<\varepsilon$. 
\end{proposition}
\begin{proof}
    We use induction on $\alpha$.
    Clearly, the statement holds for $\alpha=0$. 

    For the limit step, it is enough to observe that whenever $\gamma<\omega_1$ is a limit ordinal, $\mathrm{II} \uparrow\EFD_{\gamma,\varepsilon}(Aa,Bb)$ if and only if $\mathrm{II} \uparrow\EFD_{\alpha,\varepsilon}(Aa,Bb)$ for all $\alpha<\gamma$.

    For the successor step, note that
    \[r_{\alpha+1,n}(Aa,Bb)<\varepsilon\iff\forall c<\omega \ \exists d,d'<\omega \ \colon \ \max\big(r_{\alpha,n+1}^{A,B}(ac,bd),r_{\alpha,n+1}(Aad',Bbc)\big)<\varepsilon\]
    by the inductive definition of $r_{\alpha,n}$.
    Thus, it is enough to show that 
    \begin{align*}\label{eq:Games}
        \textrm{II} \uparrow\EFD_{\alpha+1,\varepsilon}(Aa,Bb)\iff  \forall c<\omega \ \exists d,d'<\omega \colon~ & \textrm{II} \uparrow\EFD_{\alpha,\varepsilon}(Aac,Bbd)\\
        \text{ and }& \textrm{II} \uparrow\EFD_{\alpha,\varepsilon}(Aad',Bbc).
    \end{align*}
    Suppose $(S_n)_{n\in\mathbb N}$ is a winning strategy for II in $\EFD_{\alpha+1,\varepsilon}(Aa,Bb)$. Let $c\in \omega$. Then there is some $d\in\omega$ such that $(\alpha,c,d)\in S_1$. Now, setting for each $n\in\mathbb N$,
    \[S_n':=\{(\beta,e,f)\in\omega\times\omega^n\times\omega^n \colon (\beta,ce,df)\in S_{n+1}\}\]
    defines a winning strategy for II in $\EFD_{\alpha,\varepsilon}(Aac,Bbd)$. A winning strategy for II in $\EFD_{\alpha,\varepsilon}(Aad',Bbc)$ can be found analogously.

    On the other hand, suppose for each $c<\omega$, there exists $d(c),d'(c)<\omega$ and winning strategies $(S_{c,n})_n\in\mathcal{S}_{\alpha,\varepsilon}^{Aac,Bbd}(\mathrm{II})$ and $(S_{c,n}')_n\in \mathcal{S}_{\alpha,\varepsilon}^{Aad',Bbc}(\mathrm{II})$. 
    Then setting 
    \[S_1:=\{(\beta,c,d(c)) \colon c\in \omega,\beta\in\alpha\}\cup\{(\beta,d'(c),c) \colon c\in\omega,\beta\in\alpha\}\] and 
    \[S_{n+1}:=\{(\beta,ce,d(c)f) \colon (\beta,e,f)\in S_{c,n}\}\cup \{(\beta,d'(c)e,cf) \colon ~(\beta,e,f)\in S_{c,n}'\}\]
    yields a winning strategy for $\mathrm{II} $ in $\EFD_{\alpha+1}^\varepsilon(Aa,Bb)$.
\end{proof}

In particular, denoting by $\EFD_\alpha(A,B)$ the game where Player I gets to choose some $\varepsilon=\frac{1}{n}$ in the first round, after which $\EFD_{\alpha,\varepsilon}$ is played, we obtain the following corollary.

\begin{corollary}
    For $A,B\in\Mod_\omega(L)$, $A\equiv_\alpha^{bf} B$ if and only if $\mathrm{II}\uparrow \EFD_\alpha(A,B)$.
\end{corollary}
\begin{proof}
    By definition, $A\equiv_\alpha^{bf} B$ if any only if $r_{0,\alpha}(A,B)=0$. The statement then follows immediately from Proposition \ref{proposition: Games correspond to bnf pseudo-distance} and the fact that a winning strategy for II in $\EFD_\alpha(A,B)$ is precisely a winning strategy for $\mathrm{II}$ in $\EFD_{\alpha,\varepsilon}(A,B)$ for every $\varepsilon>0$.
\end{proof}

\begin{lemma}\label{lemma: set of winning strategies is Borel}
    For any $\varepsilon>0$, $(\lambda,\prec)\in\LO$ and $A,B\in\Mod_\omega(L)$, the sets $\mathcal{S}_{\prec,\varepsilon}^{A,B}(\mathrm{I})$ and $\mathcal{S}_{\prec,\varepsilon}^{A,B}(\mathrm{II})$ are Borel in $\prod_{n\in\mathbb N}\mathcal{P}(\omega^n\times\omega^n\times\omega^n)$.
\end{lemma}
\begin{proof}
    This is immediate from the presentation of Definition \ref{definition: metric EFD winning strategy for II} and Definition \ref{definition: metric EFD winning strategy for I}, together with the fact that the maps $r_{0,n}(A,\cdot,B,\cdot) \colon \omega^n\times\omega^n\to\mathbb R^+$ are Borel by Fact \ref{fact: pseudo bnf distance is Borel}.
\end{proof}

We show now that the metric dynamic EF-games are \textit{determined}, that is, one of the two players always has a winning strategy (but not both).

\begin{lemma}\label{lem: determinacy of metric EFD}
    For any $A,B\in\Mod_\omega(L)$, $\varepsilon>0$ and $\prec\in\LO$,
    \[ \mathrm{I} \uparrow\EFD_{\prec,\varepsilon}(A,B)\iff \neg \mathrm{II} \uparrow\EFD_{\prec,\varepsilon}(A,B).\]
\end{lemma}
\begin{proof}
    This follows from the fact that by definition, $\EFD_{\prec,\varepsilon(A,B)}$ is a {closed} game. That is, even though a gameplay can have infinitely many rounds, if I wins in a gameplay, then this is decided after a finite number rounds. It is a classical fact from descriptive set theory that such games are determined, see for example \cite[Chapter II.20]{MR1321597}.
\end{proof}

For any $\varepsilon>0$, write  
\[Q_\varepsilon:=\{(A,B,\prec)\in\Mod_\omega(L)^2\times\LO \, \colon \, \mathrm{I}\uparrow\EFD_{\prec,\varepsilon}(A,B)\}\]
and
\[R_\varepsilon:=\{(A,B,\prec)\in\Mod_\omega(L)^2\times\LO \, \colon \, \mathrm{II} \uparrow\EFD_{\prec,\varepsilon}(A,B)\}.\]

\begin{lemma}\label{lemma: Winning strategy Borel}
    The set $R_\varepsilon$ is Borel.
\end{lemma}
\begin{proof}
    It follows from Lemma \ref{lemma: set of winning strategies is Borel} that the set
    \[\big\{(A,B,(\lambda,\prec),(S_n)_n) \ \colon \ (S_n)_n\in \mathcal{S}_{\prec,\varepsilon}^{A,B}(\mathrm{II})\big\}\]
    is Borel in $\Mod_\omega(L)^2\times\LO\times\prod_n(\omega^n\times\omega^n\times\omega^n)$. Since $R_\varepsilon$ is just the projection of this set to $\Mod_\omega(L)^2\times\LO$, it follows that $R_\varepsilon$ is analytic.
    By an analogous argument, $Q_\varepsilon$ is analytic.
    Since $\big(\Mod_\omega(L)^2\times\LO\big)\setminus R_\varepsilon=Q_\varepsilon$ by Lemma \ref{lem: determinacy of metric EFD}, we conclude that $R_\varepsilon$ is Borel by Suslin's theorem.
\end{proof}

We conclude with the following key result, amounting to Theorem \ref{theorem: intro uniform Borelness of games}. Note that any countable ordinal $\alpha<\omega_1$ is isomorphic to some $(\lambda,\prec)\in\LO$.

\begin{theorem}\label{theorem:back-and-forth_is_uniformly_borel}
    The set $R:=\bigcap_{\varepsilon>0}R_\varepsilon\subseteq\Mod_\omega(L)^2\times\LO$
    is Borel, and for any ordinal $\alpha<\omega_1$,
    \[A\equiv_\alpha^{bf} B\iff (A,B,\alpha)\in R.\]
\end{theorem}

\begin{proof}
    Taking $\varepsilon$ to range over $\{n^{-1} \colon n<\omega\}$, $R$ is Borel as a countable intersection of Borel sets by Lemma \ref{lemma: Winning strategy Borel}. The equivalence holds by Proposition \ref{proposition: Games correspond to bnf pseudo-distance}.
\end{proof}

\subsection{Refining families of Borel equivalence relations}\label{subsection: barwise}

Consider sets $X$ and $Y$ together with
equivalence relations $E\subseteq X^2$ and $F\subseteq Y^2$. Suppose we are given a map $K \colon X\to Y$ such that $x \ E \ x'$ holds if and only if $K(x) \ F \ K(x')$. Suppose further that $F$ is the intersection of a decreasing sequence $(F_\alpha)_{\alpha<\omega_1}$ of equivalence relations on $Y$. It is natural to ask whether classification already occurs at lower levels of the sequence of equivalence relations, i.e., whether there is some $\beta<\omega_1$ such that $x \ E \ x'$ already holds whenever $K(x) \ F_\beta \ K(x')$. One can ask a similar question in the case where $E$ is the intersection of such a sequence.
In general there is no reason this should be the case, but we show in this section that under suitable descriptive set theoretic assumptions on $X,~Y,~E,~F$ and $K$ such $\beta<\omega_1$ exists.

We use the following classical result by Lusin and Sierpiński (see \cite[Theorem 27.12]{MR1321597}). Denote by $\WO\subseteq\LO$ the set of well-orders on subsets of $\omega$. Recall that $\LO$ denotes the set of linear orders on subsets of $\omega$. Observe that $\LO$ is Borel and $\WO$ is co-analytic.

\begin{theorem}[Lusin-Sierpiński]\label{thm:DST:not_sigma_one_one_well_orders}
    The set $\WO$ is not analytic. 
\end{theorem}

By Suslin's theorem, non-analyticity of $\WO$ is a consequence of non-Borelness of $\WO$, or equivalently (by the Lopez-Escobar theorem) non-definability of $\WO$ in $L_{\omega_1, \omega}$ whenever $L=\{<,\dots\}$. This can be seen as a rather weak fraction of compactness that is preserved in infinitary logic: If $L=\{<,\dots\}$ and $\sigma\in L_{\omega_1, \omega}$ is a sentence, if $\sigma$ has models of order-type $\alpha$ for all $\alpha<\omega_1$, then $\sigma$ has a model of the order-type of $(\mathbb Q,<)$.

\begin{definition}
    Call a family $R\subseteq X\times \mathcal{I}$ indexed by $\mathcal{I}\subseteq \LO$ \textit{refining}, if $R_I\supseteq R_{I'}$ whenever $I, I'\in \mathcal{I}$ are such that $I$ embeds into $I'$.
\end{definition}

Denote by $\omega^*$ the linear order obtained by inverting the usual order on $\omega$. Denote by $X$ and $Y$ analytic subsets of respective standard Borel spaces.
As a consequence of Lusin-Sierpiński, we obtain the following, 

\begin{lemma}\label{lemma:DST:Barwise_compactness}
    Let $R\subseteq X\times\LO$ be analytic and refining. If $R_{\alpha}\neq\emptyset$ for every $\alpha<\omega_1$, then $R_{\omega^*}\neq\emptyset$. 
\end{lemma}
\begin{proof}
    Assume otherwise. Consider the set
    \[ \mathcal{I}:=\{I\in\LO \colon R_I\neq\emptyset\}=\{ I \in \LO \colon (\exists x \in X)( x \in R_{I} ) \}. \]
    By assumption this set contains all well-orders on $\omega$ and it does not contain $\omega^*$. Since any non-well-order contains a copy of $\omega^*$ and $R$ is refining, we conclude that $R_I=\emptyset$ for all non-well-orders $I$. Thus, $\mathcal{I}=\WO$. Moreover, $\mathcal{I}$ is analytic as a projection of the analytic set $R$, contradicting Theorem \ref{thm:DST:not_sigma_one_one_well_orders}.
\end{proof}

\begin{definition}\label{definition:uniformly_borel_equivalence}
    A refining family $E=(E_\alpha)_{\alpha<\omega_1}\in\mathcal{P}( X^2)\times \omega_1$ is called a \textit{refining Borel (analytic) equivalence relation on $X$} if $E_\alpha\subseteq X^2$ is a Borel (analytic) equivalence relation for each $\alpha<\omega_1$. $E$ is called \textit{uniformly Borel (analytic)} if there is a Borel (analytic) refining ${R}\subseteq X^2\times \LO$ with $E_\alpha= R_\alpha$ for each $\alpha<\omega_1$ and $R_{\omega^*}= \bigcap_{\alpha<\omega_1} E_\alpha$. In this case, we say $R$ induces $E$.
\end{definition}

For a refining family $(E_\alpha)_{\alpha<\omega_1}$, we denote $E_\infty:=\bigcap_\alpha E_\alpha$. 

\begin{example}\label{example:bnf_uniformly_Borel}
    Let $L$ be some countable metric language and consider $\equiv_\alpha^{bf}\subseteq \Mod_\omega(L)^2$ as in Definition \ref{definition:back-and-forth-distance} (with a universal weak modulus that is fixed implicitly). Then $(\equiv_\alpha^{bf})_{\alpha<\omega_1}$ is refining, and it is uniformly Borel by Theorem \ref{theorem:back-and-forth_is_uniformly_borel}. 
\end{example}

\begin{theorem} \label{thm: DST_main}
    Let $K \colon X\to Y$ be an analytic map, let $E$ be a Borel equivalence relation on $X$, and let $(F_\alpha)_{\alpha<\omega_1}$ be a uniformly analytic refining equivalence relation on $Y$. If $x \ E \  x'$ holds whenever $K(x) \ F_\infty \ K(x')$, then there exists $\beta<\omega_1$ such that $x \ E \ x'$ holds whenever $K(x) \ F_{\beta} \ K(x')$.
\end{theorem}
\begin{proof}
    We argue towards a contradiction by assuming that for every $\beta<\omega_1$ there are $x, x' \in X$ such that $\neg (x \ E \ x')$ and $K(x) \ F_\beta \ K(x')$. By assumption, there is an analytic, refining $R\subseteq Y^2\times\LO$ inducing $F$.
    Let
    \[R':=\Big\{(y,y',I)\in R \colon \exists x,x'\in X \big(K(x)=y,~ K(x')=y',~ \neg (x \ E \ x')\big)\Big\}.\]
   Then $R'$ is analytic as a projection of an analytic set, since $R$, $K$, and $\neg E$ are analytic. Moreover, $R'$ is refining and $R'_\beta\neq \emptyset$ for each $\beta<\omega_1$, by assumption and the fact that $F_\beta= R_\beta$. By Lemma \ref{lemma:DST:Barwise_compactness}, $R'_{\omega^*}$ is non-empty. Thus, there are $x, x' \in X$ such that $\neg x \ E \ x'$ and $(Kx, Kx')\in R_{\omega^*}$. However, $R_{\omega^*} = F_{\infty}$, so $K(x) \ F_{\infty} \ K(x')$ and thus $x \ E \ x'$ by assumption.
\end{proof}

Note that for the above proof, it would have been sufficient to assume co-analyticity of $E$ in place of Borelness. Moreover, it would be enough to have $F_\beta\subseteq R_\beta$ for each $\beta<\omega_1$ and $R_{\omega^*}\subseteq F_\infty$.
A converse of Theorem \ref{thm: DST_main} can be proved analogously, assuming uniform analyticity of the sequence of equivalence relations $(E_\alpha)_{\alpha<\omega_1}$, and that $K$ maps $E_\infty$-equivalent pairs to $F$-equivalent pairs.

\begin{theorem} \label{thm: DST_main_2}
    Let $K \colon X\to Y$ be an analytic map, let $(E_\alpha)_{\alpha<\omega_1}$ be a uniformly analytic refining equivalence relation on $X$, and $F$ a Borel equivalence relation on $Y$. If $K(x) \ F \ K(x')$ holds whenever $x \ E_\infty \ x'$, then there exists $\beta<\omega_1$ such that $K(x) \ F \ K(x')$ holds whenever $x \ E_{\beta} \ x'$.
\end{theorem}

Again, note that it is sufficient to assume co-analyticity of $F$, in place of Borelness.

We turn now to the case where $X$ and $Y$ are spaces of separable metric structures. Let $L_1$ and $L_2$ be metric languages. Recall that $\equiv_\alpha$ denotes elementary equivalence up to quantifier-rank $\alpha$ and $\equiv_\alpha^{\text{bf}}$ is back-and-forth equivalence in the sense of Definition \ref{definition:back-and-forth-distance} with respect to a fixed universal weak modulus in the respective language. We may write ${\equiv}^{L_i}_\alpha$ and ${\equiv}_\alpha^{L_i,\text{bf}}$ to highlight the dependence on the ambient language. As an immediate application of the above theorem and Example \ref{example:bnf_uniformly_Borel}, we obtain the following, amounting to Corollary \ref{corollary:intro comparison of Scott degrees}. In the following statement (and its proof) we identify $A \in \Mod_\omega(L)$ with the separable $L$-structure $\widehat{A}$ it encodes.

\begin{corollary} \label{cor: infinitary classification}
    Let $X\subseteq\Mod_\omega(L_1)$ and $Y\subseteq\Mod_\omega(L_2)$ be analytic and let $K \colon X\to Y$ be an analytic map such that 
    \[ \forall A,B\in X\colon \ K(A) \cong K(B) \iff A \cong B. \]
    Then there exist maps $\theta\colon\omega_1\to\omega_1, \, \theta'\colon\omega_1\to\omega_1$ such that
    \begin{equation*}
        \forall A,B\in X\ \forall\alpha<\omega_1 \colon
        \begin{cases}
            K(A) \equiv_{\theta(\alpha)} K(B)\implies A\equiv_\alpha B, \\
            A\equiv_{\theta'(\alpha)}B\implies K(A)\equiv_\alpha K(B).
        \end{cases}
    \end{equation*}
\end{corollary}
\begin{proof}
    Note that for $i\in\{0,1\}$, both $({\equiv}^{L_i}_\alpha)_{\alpha<\omega_1}$ and $\equiv_\alpha^{L_i,bf}$ are refining sequences of equivalence relations, and $\bigcap_{\alpha}{\equiv}^{L_i}_\alpha$ and $\bigcap_\alpha \equiv_\alpha^{L_i,bf}$ coincide with $\cong$. Moreover, $(\equiv_\alpha^{L_2,bf})_{\alpha<\omega_1}$ is uniformly Borel by Theorem \ref{theorem:back-and-forth_is_uniformly_borel}. Applying Theorem \ref{thm: DST_main} with $E=E_\alpha$ and Theorem \ref{thm: DST_main_2} with $F=F_\alpha$, for each $\alpha<\omega_1$, yields functions $\theta \colon \omega_1\to\omega_1$ and $\theta' \colon \omega_1\to\omega_1$ such that for any $A,B\in\Mod_\omega(L_1)$ and $\alpha<\omega_1$,
    \[K(A)\equiv_{\theta(\alpha)}^{L_2,bf}K(B)\implies A\equiv^{L_1}_\alpha B\]
    and
    \[A\equiv_{\theta'(\alpha)}^{L_1,bf}B\implies K(A)\equiv_\alpha K(B)\]
    Moreover, $\equiv_\beta^{L_i}$ refines $\equiv_\beta^{L_i,bf}$, for each $\beta<\omega_1$ by Fact \ref{fact: bnf pseudo-metric properties}. The statement follows.
\end{proof}

\subsection{Borel category of $L$-structures}\label{subsection: Borel categories}

We introduce now the notion of a \emph{Borel category}, and construct a \emph{Borel model} for the category of separable $C^*$-algebras. 
More generally, we construct a Borel model of the category of separable models of any given theory in a countable metric language. 
If one considers only surjective isometries between models, such Borel groupoids were already considered in \cite{ElliottEtAl2013, MR3660238} (using the Urysohn space) and in \cite{yaacov2017metric}. Characterizations of Borel functors between Borel groupoids of models were given in \cite{MR3893282, MR4029718}. Let us note that the study of Borel structures (see \cite{steinhorn1985chapter, MR2830412} for the definition) was initiated by Friedman \cite{friedman1979borel} and there are many results knows for some `degenerated' Borel categories, like orderings or equivalence relations \cite{MR965754, MR1057041}.

\begin{definition} \label{def: Borel category}
    A Borel category $\mathbf{X}$ is a pair $(X_{\Ob},X_{\Mor})$ of standard Borel spaces together with Borel source and target maps $s,t \colon X_{\Mor} \to X_{\Ob}$, a Borel identity map $i \colon X_{\Ob} \to X_{\Mor}$, and a Borel composition map
    \[ \circ_X \colon X_{\Mor} \times_{t, X_{\Ob},s} X_{\Mor} \to X_{\Mor}, \]
    that form a category, where $X_{\Mor}(A,B) := \{ f \in X_{\Mor} \colon s(f)=A, t(f)=B \}$ is the set of morphisms between $A,B \in X_{\Ob}$. 
\end{definition}

In other slightly informal words, a Borel category is a 1-category object in the category of standard Borel spaces. When the ambient category $\mathbf{X}$ is clear from the context, we denote composition by $\circ$ instead of $\circ_{{X}}$.

From here onward, boldface letters $\mathbf{X},\mathbf{Y},\mathbf{Z}$ will denote Borel categories and curly letters $\mathcal{C},\mathcal{D},\mathcal{X},\mathcal{Y}$ will denote categories.

\begin{definition}
    A Borel functor $F \colon \mathbf{X}\to\mathbf{Y}$ is a pair of Borel maps $(F_{\Ob},F_{\Mor})$ where $F_{\Ob} \colon X_{\Ob} \to Y_{\Ob}$, $F_{\Mor} \colon X_{\Mor} \to Y_{\Mor}$ such that $F$ defines a functor between the underlying categories.

    For Borel functors $F,G \colon \mathbf{X}\to\mathbf{Y}$, a Borel natural transformation $\alpha \colon F\to G$ is a Borel map $\alpha \colon X_{\Ob} \to Y_{\Mor}$ such that for any $A\in X_{\Ob}$, $\alpha(A)\in X_{\Mor}(F_{\Ob}(A),G_{\Ob}(A))$, and for any $f\in X_{\Mor}(A,B)$,
    \[ G_{\Mor}(f) \circ_{Y} \alpha(A) = \alpha(B) \circ_{Y} F_{\Mor}(f).\]
\end{definition}

It is clear that one could also consider higher Borel categories, however in this paper we only need the language of 1-categories.

\begin{definition}
    Let $\mathcal{C}$ be a category. A Borel category $\mathbf{X}$ is a Borel model of $\mathcal{C}$ if it is equivalent to $\mathcal{C}$ as a category. 
\end{definition}

The following is a standard fact from category theory.

\begin{fact}\label{fact: equivalences of categories}
    $\mathcal{C}$ and $\mathcal{D}$ are equivalent if and only if there is a functor $F \colon \mathcal{C}\to\mathcal{D}$ that is essentially surjective and fully faithful. 
\end{fact}

Fix some metric language $L$. We aim now to find a Borel model of the category of separable $L$-structures, with \emph{homomorphisms} in the sense of the following definition as morphisms.

\begin{definition}\label{definition: homomorphism}
    Let $M$ and $N$ be $L$-structures. A homomorphism $M\to N$ is a map $\psi \colon M\to N$ such that $\psi$ commutes with interpretations of function and constant symbols in $L$
    and for any relation symbol $R(x)$ in $L$ and $a\in M^x$, 
    \begin{equation}\label{eq: respect relation symbols}
        |R^{N}(\psi(a))|\leq |R^{M}(a)|.
    \end{equation} 

    For an $L_{\omega_1,\omega}$-theory $T$, write $\mathcal{C}_T$ for the category of separable models of $T$, with homomorphisms as morphisms.
\end{definition}

Note that in particular all homomorphisms are contractive, by (\ref{eq: respect relation symbols}) with $R$ being the distinguished symbol $d$.

\begin{remark}
    The definition of homomorphism above is consistent with that in discrete logic, where a homomorphism is one that commutes with function- and constant symbols, and where 
    \[a\in R^{M}\implies \psi(a)\in R^N,\]
    which coincides with (\ref{eq: respect relation symbols}) if we interpret $R^M$ as a $\{0,1\}$-valued map with $R(a)=0$ if and only if $a\in R$.
\end{remark}

We fix a countable collection of $L_{\omega_1,\omega}$-sentences $\theory$ for some metric language $L$.
By Lemma \ref{lem: set of separable models is Borel} and the fact that $T$ is countable, the set $\Mod_\omega(\theory) \subset \Mod_\omega(L)$ is Borel.
The goal now is to construct a Borel model $\mathbf{X}^{\theory}$ of $\mathcal{C}_{\theory}$  with $X^{\theory}_{\Ob}:=\Mod_\omega(\theory)$, i.e. to find a functor $F \colon \mathbf{X}^{\theory} \to\mathcal{C}_{\theory}$ witnessing equivalence of categories. We use the convention that if we talk about functors from Borel categories to usual categories we mean functors from the category underlying the Borel one.
For $A\in X_{\Ob}$, we will denote by $\widehat{A}$ the separable $L$-structure that $A$ encodes, i.e., $\widehat{A}$ is the set of equivalence classes of Cauchy sequences in $A$, with functions and relations extended continuously (so until the end of this section we do \textbf{not} identify $A$ with $\widehat{A}$).
At the level of objects, it is clear that $F$ will coincide with $\widehat{\cdot}$, which we have already seen to be essentially surjective (i.e., every separable model can be encoded in $\Mod_\omega(\theory)$). 
We are left with finding a suitable space of morphisms in $\mathbf{X}^{\theory}$. For $A,B\in \Mod_\omega(\theory)$, we want to encode homomorphisms $\widehat{f} \colon \widehat{A}\to\widehat{B}$ in a way that only depends on $A$ and $B$. We do this in the natural way, using Cauchy sequences.

For notational convenience, we write $\bigwedge, \bigvee$ for countable conjunction and disjunction respectively; when we quantify in this way over $\varepsilon>0$ (or $\delta>0$, etc.) we actually mean quantification over $\varepsilon = 1/k$ for natural numbers $k$. Recall that the universe of each $A\in X^{\theory}_{\Ob}$ is $\omega$ so the expressions $a\in A$ and $a\in\omega$ are synonymous (but the former indicates how one thinks of $a$ as an element of the structure $A$). We write below $f$ for a sequence $(f_n)_{n\in\omega}\in (\omega^\omega)^\omega$.

\begin{lemma}\label{lemma: Cauchy sequences X_1 Borel}
    Let $X_1 \subset X^{\theory}_{\Ob} \times X^{\theory}_{\Ob} \times (\omega^\omega)^\omega$ be the set of triples $(A,B,f)$ satisfying the following conditions.
    \begin{itemize}
        \item[(i)] For every $a \in A$, the sequence $f_n(a)$ is Cauchy with respect to $d_B$;
        \item[(ii)] There is a (unique) homomorphism $\widehat{f} \colon \widehat{A} \to \widehat{B}$ such that $\widehat{f}(a)=\lim_n f_n(a)\in \widehat{B}$ for each $a\in A$.
    \end{itemize}
    Then $X_1$ is a Borel subset of the ambient product. 
\end{lemma}
\begin{proof}
     Let $(A,B,(f_n)_{n\in\omega})\in X^{\theory}_{\Ob}\times X^{\theory}_{\Ob}\times {(\omega^\omega)}^{\omega}$. We check that the property $(A,B,(f_n)_n)\in X_1$ is indeed Borel.
    \begin{itemize}
        \item[(i)] The property that $(f_n(a))_{n\in\omega}$ is Cauchy for every $a\in A$ is captured by the formula
        \[ \bigwedge_{a\in A}\bigwedge_{\varepsilon >0} \bigvee_{N<\omega} \bigwedge_{n,m > N} d_B(f_n(a),f_m(a)) < \varepsilon \]
        which is Borel as we only used countable intersections and unions and the condition $d_B(f_n(a),f_m(a))<\varepsilon$ is open. 
        \item[(ii)] Let $R$ be some relation symbol of arity $k$. We check that $\widehat{f}$ respecting $R$ in the sense of (\ref{eq: respect relation symbols}) is a Borel condition on $(A,B,f)$. The condition is captured by the formula
        \[ \bigwedge_{a\in A^k} \bigwedge_{\varepsilon > 0}\bigvee_{N<\omega}\bigwedge_{n\geq N} |R^B(f_n(a))|<|R^A(a)|+\varepsilon, \]
        which is clearly Borel. Note it is crucial here that $R^{\widehat{B}}$ is continuous, i.e. that $R^{\widehat{B}}(\widehat{f}(a))=\lim_n R^B(f_n(a))$. Similarly, for a $k$-ary function symbol $\rho$, the property that $\widehat{f}$ commutes with $\rho$ is defined by
        \[\bigwedge_{a\in A^k}\bigwedge_{\varepsilon > 0}\bigvee_{N<\omega}\bigwedge_{n\geq N} d_B\big(f_n(\rho^A(a)),\rho^B(f_n(a))\big)<\varepsilon,\]
        which is Borel in $(A,B,f)$.
        Note that we again used continuity of $\rho^{\widehat{B}}$; and the fact that $A$ is closed under $\rho^A$ for well-definedness of $f_n(\rho^A(a))$.
    \end{itemize}
    The above conditions are not just necessary, but also sufficient for $(A,B,(f_n)_{n\in\omega})\in X_1$. Indeed, any $(f_n)_{n\in\omega}$ with these properties induces a map $f' \colon A\to\widehat{B}$ with $f'(a):=\lim_n f_n(a)$, and this map is continuous (in fact contractive) with respect to $d_A$ and $d_B$, as can be seen by applying $(ii)$ to $R=d$. It follows that $f'$ can uniquely be extended to a continuous $\widehat{f} \colon \widehat{A}\to\widehat{B}$, and this will be a homomorphism by (ii) and density of $A$.
\end{proof}

It is clear that whenever $A,B\in X_{\Ob}^{\theory}=\Mod_\omega(L)$, every morphism of $L$-structures $\widehat{A}\to\widehat{B}$ is represented by some $f\in X_1(A,B)$. Indeed, we can pick for each $a\in A$ a sequence $(f_n(a))_{n\in\omega}\in B^\omega$ converging to $\widehat{f}(a)$. However, the choice of $(f_n(a))_{n\in\omega}$ is highly non-canonical. 
This is remedied by the following lemma.

\begin{lemma}\label{lemma:Borel_cat_of_L_structures}
    Let $\pi \colon X_1\to X^{\theory}_{\Ob}\times X^{\theory}_{\Ob} \times \mathbb{R}^{\omega \times \omega}$ be the map taking $(A,B,f)$ to $(A,B,d_f)$ where for any $a,b\in\omega$,
    \[d_f(a,b):=d_B(\lim_n f_n(a),b).\]
    Then $\pi$ is a Borel map and $\pi((A,B,f))=\pi((A,B,g))$ if and only if $\widehat{f}=\widehat{g}$. Moreover, the image $\pi(X_1)$ is a Borel set.
\end{lemma}
\begin{proof}
    To check that $\pi$ is a Borel map it is enough to see that its graph is Borel.
    Now, simply observe that $\pi(A,B,f)=(C,D,d_f)$ if and only if 
    \[A=B~\wedge~B=D~\wedge ~ \bigwedge_{a\in A}\bigwedge_{b\in B}\bigwedge_{\varepsilon>0}\bigvee_{N<\omega}\bigwedge_{n\geq N}\big|d_f(a,b)-d_B(f_n(a),b)\big|<\varepsilon,\]
    which is a Borel condition. The second assertion follows, since $\lim_n f_n(a)=\lim_n g_n(a)$ in $\widehat{B}$ if and only if $d_B(\lim f_n(a),b)=d_B(\lim g_n(a),b)$ for all $b\in B$, by density of $B$ in $\widehat{B}$.

    It follows also that $\pi(X_1)$ is analytic, since $X_1$ is Borel by Lemma \ref{lemma: Cauchy sequences X_1 Borel} and $\pi$ is a Borel map. By Suslin's theorem (Fact \ref{fact: suslin's thm}), we are left with checking that $\pi(X_1)$ is co-analytic.

    Let $(A,B,\delta)\in X^{\theory}_{\Ob}\times X^{\theory}_{\Ob}\times {\mathbb R}^{\omega\times\omega}$. Then $(A,B,\delta)\notin \pi(X_1)$ if and only if $(A,B,\delta)$ satisfies one of the following conditions.
    \begin{itemize}
        \item $\delta$ does not induce a function. The first case is that there is $a\in A$, and Cauchy sequences $(b_n),(c_n)\in B^\omega$ w.r.t. $d_B$ with $\lim_n b_n\neq\lim_n c_n$ and $\lim_n\delta(a,b_n)=\lim_n\delta(a,c_n)=0$.
        The second case is that there is some $a\in A$ that does not map to any element in $\widehat{B}$, that is, 
        \[\bigvee_{\epsilon>0}\bigwedge_{b\in B}\delta(a,b)>\epsilon.\]
        Note that both of these expressions are Borel conditions in $(A,B,\rho)$ and $a,(b_n),(c_n)$. The set of Cauchy sequence with respect to $B$ is Borel in $\omega^\omega$\footnote{This can be seen via the same proof as in Lemma \ref{lemma: Cauchy sequences X_1 Borel}.}, and similarly the condition that $\lim_nb_n\neq \lim_n c_n$ in $\widehat{B}$ is Borel. Thus, the set of triples $(A,B,\delta)$ satisfying one of these two cases is analytic, as the projection (coming from existential quantification over the Cauchy sequences) of a Borel set. 
        \item $\delta$ corresponds to a function that does not induce a homomorphism of ${L}$-structures as in Definition \ref{definition: homomorphism}. Consider, for example, a binary relation $R$. Then the function induced by $\delta$ fails to satisfy (\ref{eq: respect relation symbols}) if there are $a,a'\in A$ and Cauchy sequences $(b_n)$ and $ (b_n')$ in $B$ with $\lim_n\delta(a,b_n)=\lim_n\delta(a',b_n')=0$
        and
        \[ \bigvee_{\varepsilon > 0} \bigwedge_{N<\omega} \bigvee_{n>N} |R(b_n, b'_n)|\geq |R(a,a')| +\varepsilon .\]
        Again, this condition is analytic as a projection of a Borel set. The case of a function symbol or constant symbol is similar. We take the countable union over all symbols in $L$ of these conditions, which is still analytic. 
    \end{itemize}
    We conclude that the complement of $\pi(X_1)$ in $X^{\theory}_{\Ob}\times X^{\theory}_{\Ob}\times{\mathbb R}^{\omega\times\omega}$ is analytic, so $\pi(X_1)$ is Borel.
\end{proof} 

The map $\pi$ is precisely the quotient map under the equivalence $(A,B,f)\sim(A,B,g) \iff \widehat{f}=\widehat{g}$, and the above lemma verifies that the quotient is itself a standard Borel space, and $\pi$ is a Borel map. In the language of Borel equivalence relations, this means that $\sim$ is \textit{smooth} (see for example \cite[Definition 3.11]{Kechris_2024}).
We have successfully isolated a standard Borel space -- namely $\pi(X_1)$ -- that is in natural bijection with the family of morphisms in $\mathcal{C}_{\theory}$, and will serve as the space of morphisms in the Borel category $\mathbf{X}^{\theory}$. Thus, we write $X^{\theory}_{\Mor}:=\pi(X_1)$. We define $s,t \colon X^{\theory}_{\Mor} \to X^{\theory}_{\Ob}$ be the first and second projections respectively, and let $i \colon X^{\theory}_{\Ob} \to X^{\theory}_{\Mor}$ be the map sending $A$ to $\pi(A,A,i_A)$, where $i_A(a,b):=d_A(a,b)$. We define $\circ \colon X^{\theory}_{\Mor} \times_{t, X^{\theory}_{\Ob},s} X^{\theory}_{\Mor} \to X^{\theory}_{\Mor}$ to come from composition of homomorphisms in $\mathcal{C}_{\theory}$.

\begin{proposition}\label{proposition:Borel_category_of_L_structures}
    $\mathbf{X}^{\theory}=(X^{\theory}_{\Ob},X^{\theory}_{\Mor})$ forms a Borel category with the maps $s, t, i, \circ$ defined above.
\end{proposition}
\begin{proof}
    We proved that $X^{\theory}_{\Ob}=\Mod_\omega(\theory)$ is Borel in Lemma~\ref{lem: set of separable models is Borel} and that $X^{\theory}_{\Mor}$ is Borel in Lemma \ref{lemma:Borel_cat_of_L_structures}. Moreover, the maps $s,t,i$ are clearly Borel. It is also clear that this data forms a category. The only things left to check is
    that the composition map $\circ \colon X^{\theory}_{\Mor}\times_{t, X^{\theory}_{\Ob},s} X^{\theory}_{\Mor} \to X^{\theory}_{\Mor}$ is Borel.
    This follows once we observe that for $(A,B,\rho),(B,C,\mu),(A,B,\tau)\in X^{\theory}_{\Mor}$ we have $(A,C,\tau)=(B,C,\mu)\circ(A,B,\rho)$ if and only if
    \[\bigwedge_{a\in A}\bigwedge_{\varepsilon>0}\bigvee_{b\in B}\bigvee_{c\in C}\max\big(\rho(a,b),\mu(b,c),\tau(a,c)\big)<\varepsilon.\]
\end{proof}

We shall denote the Borel category $\mathbf{X}^{\theory}$ from Proposition \ref{proposition:Borel_category_of_L_structures} by $\bfMod_\omega(\theory)$ and call this the Borel category of separable models of $\theory$. This is justified by the following.

\begin{proposition} \label{proposition: Borel model of the category of sigma}
        $\bfMod_\omega(\theory)$ is a Borel model of $\mathcal{C}_{\theory}$.
\end{proposition}
\begin{proof}
    Let $\widehat{\,\cdot\,}\colon \Mod_{\omega}(\theory) \to \mathcal{C}_{\theory}$ take any $A\in X^{\theory}_{\Ob}$ to $\widehat{A}$ and any $(A,B,\delta)\in X^{\theory}_{\Mor}(A,B)$ to the unique ${L}$-homomorphism $\widehat{f} \colon \widehat{A} \to \widehat{B}$ with $\delta(a,b)=d_B(\widehat{f}(a),b)$ for each $a\in A$ and $b\in B$. We have constructed $s,t,i$ and $\circ$ precisely such that this is a functor. Moreover, for fixed $A,B\in X^{\theory}_{\Ob}$, $\widehat{\,\cdot\,}$ is bijective from $X^{\theory}_{\Mor}(A,B)$ to the set of $L$-homomorphisms by construction (see Lemma \ref{lemma:Borel_cat_of_L_structures}). In other words, $\widehat{\,\cdot\,}$ is fully faithful.

    Finally, $\widehat{\,\cdot\,}$ is essentially surjective, since every separable model of $\theory$ is represented by some $A\in X^{\theory}_{\Ob}$, which can be seen by picking a dense countable subset and closing it under all function symbols of $f$, as described above Definition \ref{definition: Mod_omega(L)}. The statement follows by Fact \ref{fact: equivalences of categories}.
\end{proof}

\subsection{Universal functors from Borel categories}\label{subsection: analytic conditions}

Fix a Borel category $\mathbf{X} = (X_{\Ob},X_{\Mor},s,t,i,c)$ as in Definition \ref{def: Borel category}. The goal of this part of the paper is to prove Theorem \ref{thm:T_equivalence_is_analytic} from which later, the analyticity of $KK$-equivalence will be deduced.

\begin{construction} \label{construction: describing L_0 depending on X}
    Consider a multisorted (discrete) first-order language $L_0$ such that for every $x,y \in X_{\Ob}$ there is a sort $\Mor(x,y)$ and for every $x,y,z$ there is a function symbol
    \[ c_{x,y,z}:\Mor(x,y) \times \Mor(y,z) \to \Mor(x,z) \]
    and, for every $f \in X_{\Mor}$, there is a constant $[f]$ in the sort $\Mor(s(f),t(f))$.

    Let $T_0$ be the $L_0$-theory saying that
    \begin{enumerate}
        \item the compositions $c_{x,y,z}$ are associative,
        \item composing with $[i(x)] \in \Mor(x,x)$ is the identity,
        \item the relation $c_{x,y,z}([f],[g]) = [c(f,g)]$ holds.
    \end{enumerate}
    Note that there is a one-to-one correspondence between models $M$ of $T_0$, and functors $F:\mathbf{X}\to\mathcal{C}$ with $\Ob(\mathcal{C})=X_{\Ob}$ and $F$ preserving objects, given by $\Mor_\mathcal{C}(x,y):=\Mor(x,y)^M$ for $x,y\in X_{\Ob}$ and $F(f):=[f]^M$ for $f\in X_{\Mor}$.
\end{construction}

\begin{definition} \label{definition: X-expansions}
    Let $L$ be an expansion of $L_0$ defined as 
    \[ 
    L_0 \cup \{ R_{x_1, \dots, x_{2a_i}, f_1, \dots, f_{b_i}}^i : i \in I_R\} \cup \{ F_{x_1, \dots, x_{2d_i+2}, f_1, \dots, f_{e_i}}^i : i \in I_F\} \cup \{ c_{x_1, \dots, x_{g_i}, f_1, \dots, f_{h_i}}^i : i \in I_c\},
    \]
    where $I = I_R \cup I_F \cup I_c$ is a standard Borel space split into three Borel parts, $a_i, b_i, d_i,e_i, g_i,h_i \in \mathbb{N}$ are Borel sequences, and $x_a$'s and $f_b$'s vary over elements of $X_{\Ob}$ and $X_{\Mor}$ respectively.
    For each $i \in I_R$, the symbols $R_{x_1, \dots, x_{2a_i}, f_1, \dots, f_{b_i}}^i$ are relation symbols with domain equal $\prod_{l=1}^{a_i} \Mor(x_{2l-1}, x_{2l})^{k_i}$ for some Borel sequence $(k_i)_{i \in I_R}$. For $i \in I_F$ the symbols $F_{x_1, \dots, x_{2d_i+2}, f_1, \dots, f_{e_i}}^i$ are function symbols with domain equal $\prod_{l=1}^{2d_i} \Mor(x_{2l-1}, x_{2l})^{m_i}$, for some Borel sequence $(m_i)_{i \in I_F}$, and codomain equal $\Mor(x_{2d_i+1}, x_{2d_i+2})$. 
    Similarly, for $i \in I_c$,  $c_{x_1, \dots, x_{g_i}, f_1, \dots, f_{h_i}}^i$ are constants of the sort $\Mor(x_1, x_2)$.
    
    We call such $L$ an $\mathbf{X}$\textbf{-expansion} of $L_0$.
    Let $T$ be an $L$-theory. We call $T$ a \textbf{Borel condition} (for $\mathbf{X}$), if as a set of sentences in $L$ it is Borel. Similarly, we say that $T$ is an \textbf{analytic condition}, if the underlying set of sentences in $L$ is analytic.
\end{definition}

Let us clarify the above definition. Note that an $\mathbf{X}$-expansion $L$ of $L_0$ has sorts indexed by pairs of elements from $X_{\Ob}$ and relation symbols indexed either by elements of $X_{\Ob}^3$ (namely $c_{x,y,z}$), or by products $I \times X_{\Mor}^{<\omega} \times X_{\Ob}^{<\omega}$ (e.g. $R_{x_1, \dots, x_{n_i}, f_1, \dots, f_{m_i}}^i$ or constants $[f]$ for each ${f \in X_{\Mor}}$). To form $L$-formulas we need the symbols from $L$, finitely many logical symbols $=,\neg, \wedge, \vee, \rightarrow, (, ), \forall, \exists$, and variables indexed by sorts $\Mor(x,y)$. Moreover, the set of $L$-formulas forms a Borel subset of the set of all strings of symbols and variables from the previous line (we skip the inductive proof of this fact, similar to the proof of Lemma \ref{lemma:set_of_proofs_is_Borel}). Hence, the set of $L$-formulas has a natural Borel structure making it a standard Borel space. With respect to this structure we define $T$ to be a Borel condition. Note that the subset of $L$-sentences is also Borel.

\begin{example}\label{example:Borel_condition}
    Let $L = L_0 \cup \{0_{x,y}, +_{x,y}:x,y \in X_0\}$ with $0_{x,y} \in \Mor(x,y)$ and $+_{x,y}:\Mor(x,y)^2 \to \Mor(x,y)$. Let $T$ consist of $T_0$ and axioms asserting that $(\Mor(x,y), +_{x,y}, 0_{x,y})$ is an abelian group, e.g.,
    \[ \bigl( \forall \alpha \in \Mor(x,y) \bigr)\bigl( +_{x,y}(\alpha, 0_{x,y}) = \alpha \bigr) \in T, \]
    Note that by writing $\alpha \in \Mor(x,y)$ we do not mean a set quantifier, but rather the indication of the sort over which the quantifier is taken. This and other axiom schemes in $T$ are Borel as sets of $L$-sentences, so $T$ is a Borel condition.
\end{example}

\begin{lemma}\label{lemma:set_of_proofs_is_Borel}
    Let $T$ be a Borel [resp. analytic] condition in an $\mathbf{X}$-expansion $L$ of $L_0$. Let $S$ be the set of finite sequences of $L$-formulas equipped with a natural standard Borel space structure. Let $P \subset S$ be the set of sequences that represent proofs in $T$. Then $P$ is a Borel [resp. analytic] subset of $S$.
\end{lemma}
\begin{proof}
    Every proof is a finite sequence of formulas that is built inductively as follows. First we consider a set of formulas $F_1$ that consists of:
    \begin{itemize}
        \item axioms from $T$;
        \item first-order logic tautologies, i.e., formulas that can be obtained by taking a tautology of propositional logic and replacing each propositional variable by a first-order formula;
        \item equality axioms, i.e., reflexivity, symmetry, transitivity of equality, and axioms of the form 
        \[ (\forall x,x')(x=x' \rightarrow f(x)=f(x')), (\forall x,x')(x=x' \rightarrow (R(x) \leftrightarrow R(x')),\]
        where $f$ vary among function symbols, $R$ among relation symbols, and the arities are chosen appropriately;
        \item quantifier axioms $\forall x \varphi(x) \rightarrow \varphi(t), \psi(t) \rightarrow \exists x \psi(x)$ for any term $t$.
    \end{itemize}
    This set is a Borel [resp. analytic] subset of the set of all formulas, as the first bullet is Borel [resp. analytic] by the assumption, and the latter ones are Borel by an induction argument on their length.
  
    We define inductively sets $F_n$ of proofs of length $n$ in the following manner. A sequence $(\varphi_1, \dots, \varphi_n, \varphi_{n+1})$ is in $F_{n+1}$ if $(\varphi_1, \dots, \varphi_n)$ is in $F_n$ and one of the following holds:
    \begin{itemize}
        \item (axioms) $\varphi_{n+1} \in F_1$,
        \item (modus ponens) there are $i, j \leq n$ such that $\varphi_j = (\varphi_i \rightarrow \varphi_{n+1})$,
        \item (inference rules) there is $i \leq n$ such that $\varphi_i = (\alpha \rightarrow \beta)$ (or $\varphi_i = (\beta \rightarrow \alpha)$) and $x$ is a variable not free in $\alpha$, with 
        \[ \varphi_{n+1} = (\alpha \rightarrow \forall x \beta) \textnormal{ or respectively } \varphi_{n+1} = (\exists x \beta \rightarrow \alpha). \]
    \end{itemize}
    Using the fact that the map sending a formula to the set of its free variables is Borel (which can be proved inductively), we get that if $F_n$ is Borel [resp. analytic], then so is $F_{n+1}$. As $P$ is the union of all $F_n$'s, this finishes the proof.
\end{proof}

\begin{theorem}\label{thm:T_equivalence_is_analytic}
    Let $T$ be an analytic condition in an $\mathbf{X}$-expansion $L$ of $L_0$. Consider the equivalence relation $E$ of pairs $(x,y) \in X_{\Ob}^2$ that satisfy
    \begin{equation}\label{equation:T_models_equivalence_of_x_y}
        T \models (\exists \alpha \in \Mor(x,y))(\exists \beta \in \Mor(y,x))( \alpha \circ \beta = [i(y)] \wedge \beta \circ \alpha = [i(x)] ),
    \end{equation}
    where $\circ$ are the compositions $c_{x,y,x}, c_{y,x,y}$ respectively. Then $E$ is an analytic subset of $X_{\Ob}^2$.
\end{theorem}
\begin{proof}
    By Gödel's completeness theorem, the condition defining $E$ is equivalent to the existence of a (formal) proof of the sentence (\ref{equation:T_models_equivalence_of_x_y}) from the axioms in $T$. Every proof has finite length, so it is enough to check that the set of proofs of length $n$ that finish on (\ref{equation:T_models_equivalence_of_x_y}) is analytic. This follows from Lemma~\ref{lemma:set_of_proofs_is_Borel}.
\end{proof}

More generally, the same proof gives the following.

\begin{theorem} \label{theorem: abstract analytic condition makes analytic relation}
    Let $T$ be an analytic condition in an $\mathbf{X}$-expansion $L$ of $L_0$. Let $\varphi = \varphi_{i, x, f}$ be an $L$-sentence and let $(i,x,f) \in I^{<\omega} \times X_{\Ob}^{<\omega} \times X_{\Mor}^{<\omega}$ be all the elements that appear in $\varphi$ if we treat it as a string of symbols (and we identify variables in $\Mor(x,y)$ with elements of $(x,y,n) \in X_{\Ob}^2 \times \omega$). 
    
    Consider the subset $E$ of tuples $(i',x',f') \in I^{<\omega} \times X_{\Ob}^{<\omega} \times X_{\Mor}^{<\omega}$ of the same respective lengths as $(i,x,f)$ such that $\varphi_{i', x', f'}$ (i.e., $\varphi$ with $(i,x,f)$ replaced by $(i',x',f')$) is an $L$-sentence and
    \begin{equation*}\label{equation:T_models_equivalence_of_x_y_new}
        T \models \varphi_{i', x', f'}.
    \end{equation*}
    Then $E$ is analytic.
\end{theorem}

Since our main application is analyticity of the $KK$-equivalence relation, we do not strive for the most general form of the preceding theorem. It seems that there should be a `mixed' logic, with some sorts being infinitary, and some first-order. In such logic, the above result could be interpreted as: if $\varphi(x)$ is a first-order formula with only infinitary free variables $x$, then the set of its realisations is analytic (provided we start with an analytic theory). It would be interesting to see a direct development of such logic and whether it has other applications.

\section{Application to $C^*$-algebras}\label{section: Application}
The desire to classify objects is ubiquitous across mathematical disciplines; here we discuss one particular classification program for norm-closed subalgebras of bounded operators of Hilbert spaces, i.e., $C^*$-algebras. 
See \cite[\S~1.1, \S~1.2]{CGSTW-new-classificationI-2023} for a survey and brief history of this topic. 
We assume familiarity with $C^*$-algebras, whose model theory is laid out in \cite[\S~2.3.1]{farah2014model}, see also \cite{farah2021MTofC*}. 
We start by presenting the unital classification theorem along with the classifying invariant, in order to set out a first-order language suitable to understanding the invariant.

\subsection{Borel parametrizations of $C^*$-algebras}\label{subsec: borel parametrization of C*}

Kechris \cite{Kechrisdescriptivecstar} introduced a standard Borel space parametrization of the space of separable $C^*$-algebras. Farah, Toms and Törnquist \cite{farah2013descriptive, Farahturbulence2014} considered a few other such parametrizations, compared them, and studied Borel complexity of several equivalence relations and invariants appearing in this context. Furthermore, in \cite{farah2021MTofC*} criteria on Borelness of some classes of $C^*$-algebras are given. We survey these results and introduce some notations below.

We follow the presentation of \cite[Section~2.1]{Farahturbulence2014} for the standard Borel structure associated to a given $C^*$-algebra. 
In particular, let $H$ be a separable infinite dimensional Hilbert space and let $\mathcal{B}(H)$ be the space of bounded operators on $H$; note that $\mathcal{B}(H)$ is a standard Borel space when equipped with the Borel structure generated by the weakly open subsets. 
For every $\gamma \in \mathcal{B}(H)^\mathbb{N}$, we can associate $C^*(\gamma)$, the $C^*$-algebra generated by the sequence $\gamma$.
Then since all separable $C^*$-algebras arise as $C^*$-subalgebras of $\mathcal{B}(H)$, the space $\Gamma := \mathcal{B}(H)^\mathbb{N}$ gives a standard Borel parametrization for all separable $C^*$-algebras.
We denote by $\Gamma_u$ the subset consisting of parameters for unital separable $C^*$-algebras.

On the other hand, let $\widehat{\Xi} \subset \omega^{\omega^2} \times \omega^{\omega^2} \times \omega^{\omega} \times \mathbb{R}_{\geq 0}^{\omega} \times (\omega^{\omega})^{\mathbb{Q}(i)}$ be the subset of tuples $A =(+,\times, (-)^*,\|-\|,(\lambda \cdot)_{\lambda \in \mathbb{Q}(i)})$ that define a $C^*$-algebra. 
For the details see \cite[Section 2.4]{Farahturbulence2014}. 
Each $A \in \widehat{\Xi}$ defines a $C^*$-algebra $\widehat{A}$ which is the completion of $\omega$ with respect to the norm $\|-\|$. 
Note that $\widehat{\Xi}$ is a Borel subset of the ambient product, so it is a standard Borel space. 
Farah, Toms, and Törnquist prove that there is a Borel isomorphism $f \colon \Gamma \setminus \{0\} \to \widehat{\Xi}$ such that for every $\gamma \in \Gamma \setminus \{0\}$, the $C^*$-algebra $C^*(\gamma)$ is isomorphic to $\widehat{A}$ for $A = f(\gamma)$. 
Thus, any of these standard Borel spaces can be seen as parametrizing non-trivial separable $C^*$-algebras. These parametrizations are equivalent to the one we discussed in Section \ref{subsection: Borel categories} in the following sense.

\begin{lemma}\label{lemma:our presentation of C* is equivalent to the standard one}
    Let $T_{C^*}$ be the theory from Example \ref{example:language of C*}. There is a Borel isomorphism $g:\Mod_{\omega}(T_{C^*}) \to \widehat{\Xi}$ such that $\widehat{A} \cong \widehat{g(A)}$. 
\end{lemma}
\begin{proof}
    An element $A \in \Mod_{\omega}(T_{C^*})$ gives a choice of a dense countable subset in the ball of radius $n$ in $\widehat{A}$, for each $n<\omega$, which is additionally closed under functions (that may increase the ball radii). 
    By taking the union of all these subsets, we get a countable dense subset $g(A) \subseteq \widehat{A}$ closed under all required operations. 
    This finishes the proof. 
\end{proof}

We again state the unital classification theorem, and discuss why the set of simple nuclear $\cZ$-stable $C^*$-algebras form a Borel subset of any standard Borel space of unital separable $C^*$-algebras.

\begin{theorem} \label{thm: classification of C*-algebras}
Unital simple separable nuclear $\cZ$-stable $C^*$-algebras satisfying Rosenberg and Schochet’s universal coefficient theorem are classified by the invariant $KT_u$ consisting of $K$-theory and traces.
\end{theorem}

Let us a take a brief moment to consider the assumptions of this statement. This class of $C^*$-algebras was briefly discussed in the introduction to this paper; each assumption here is necessary for classification,\footnote{The open question asking if the UCT holds for all nuclear $C^*$-algebras could render the UCT assumption redundant.} and a deeper discussion of each of the assumptions can be found in the introduction of \cite{CGSTW-new-classificationI-2023}. 
For our purposes, recall that our Borel parametrizations are only for separable $C^*$-algebras. 
Simplicity, nuclearity, and $\cZ$-stability were previously studied from a descriptive set-theoretic viewpoint. 
The presence of $\cZ$-stability is Borel computable by \cite[Theorem 1.1(vi)]{farah2013descriptive}, while simplicity and nuclearity are model-theoretically accessible: they are definable by uniform families of formulas \cite[Theorem 5.7.3(3), (6)]{farah2021MTofC*}, and therefore Borel-computable by \cite[Proposition 5.17.1]{farah2021MTofC*}. This establishes the following.

\begin{theorem} \label{thm: USSN Z-stable is Borel}
    The set of separable unital $C^*$-algebras satisfying any combination of the following properties is a Borel set: simple, nuclear, $\cZ$-stable.
\end{theorem}

We now turn our attention to the assumption of Rosenberg and Schochet's UCT \cite{RosenbergSchochetUCT}. 
In their original formulation, the UCT relates the Kasparov group $KK(A, B)$ with the $K$-theory of $A$ and $B$. For the purposes of this paper, we only require the following fact from \cite[Section 7]{RosenbergSchochetUCT}.

\begin{fact} \label{fact: UCT KK-equiv-to-commutative}
    Fix a separable $C^*$-algebra $A$. Then $A$ satisfies the UCT if and only if it is $KK$-equivalent to a commutative $C^*$-algebra.
\end{fact}

As such we now turn our attention to $KK$-equivalence (that we introduce in Definition \ref{definition: KK-equivalence}).
For separable $C^*$-algebras $A, \ B$, the abelian groups $KK(A, B)$ were originally defined by G.G. Kasparov in \cite{Kasparov1981-origKK}; these can be vaguely thought of as a group of generalized homomorphisms between $C^*$-algebras. 
These $KK$-groups have since been studied from a number of viewpoints, but we only require the category-theoretic characterization obtained by Higson in \cite{Higson-KKtheory}.

In this section, we use letters like $\mathcal{A}, \mathcal{C}$ for categories rather than the boldface letters $\mathbf{A}, \mathbf{C}$ that denote Borel categories. Let $\mathcal{C}^*$ denote the category of separable $C^*$-algebras with *-homomorphisms as morphisms (without any Borel structure). In the following definition we use the notion of homotopy between morphisms of $C^*$-algebras. For more details on this notion see the discussion preceding Construction \ref{constr: tensor with C[0,1]}.
\begin{definition} \label{def: KK-functors}
    Let $\mathcal{A}$ be an additive category, and let $F \colon \mathcal{C}^* \rightarrow \mathcal{A}$ be a functor. We call $F$ a $\kappa \kappa$-functor if it satisfies the following three properties:
\begin{enumerate}
    \item $F \colon \mathcal{C}^* \rightarrow \mathcal{A}$ is a homotopy functor, i.e., homotopic morphisms in $\mathcal{C}^*$ are sent to equal morphisms;
    \item $F$ is stable, i.e., the morphism $e_* \colon F(B) \rightarrow F(B\otimes \fancyK)$, induced by the $*$-homomorphism $b \mapsto b \otimes e$ for any rank-one projection $e \in \fancyK$, is invertible;
    \item $F$ is a split exact functor, i.e., if
    \[
        0 \rightarrow J \rightarrow D \rightleftarrows D/J \rightarrow 0 
    \] is a split exact sequence of separable $C^*$-algebras, then the sequence
    \[
        0 \rightarrow F(J) \rightarrow F(D) \rightleftarrows F(D/J) \rightarrow 0
    \]
    is also split exact.
\end{enumerate}
\end{definition}

\begin{theorem}[Theorem 4.5 in \cite{Higson-KKtheory}] \label{thm: KK}
    There is an additive category $\mathcal{KK}$ and a $\kappa \kappa$-functor $KK \colon \mathcal{C}^* \rightarrow \mathcal{KK}$ that is universal among $\kappa \kappa$-functors. More precisely, by universal we mean that for any additive category $\mathcal{A}$ and any $\kappa \kappa$-functor $F \colon \mathcal{C}^* \rightarrow \mathcal{A}$, there is a unique $\widehat{F} \colon \mathcal{KK} \rightarrow \mathcal{A}$ such that $\widehat{F} \circ KK = F$. Furthermore, for $C^*$-algebras $A, B$, the abelian group of morphisms between their images in the category $\mathcal{KK}$ is naturally isomorphic to the Kasparov group $KK(A,B)$, and composition of morphisms corresponds to the Kasparov product.
\end{theorem}

\begin{definition} \label{definition: KK-equivalence}
    We call two $C^*$-algebras \textit{$KK$-equivalent}, if they become equivalent after passing to the $\mathcal{K}\mathcal{K}$ category. Stated differently, this means that their images by any $\kappa\kappa$-functor become equivalent.
\end{definition}

\subsection{The invariant classifying $C^*$-algebras}\label{subsec: invariant} 
For a unital separable $C^*$-algebra $A$, the invariant $KT_u(A)$ given in the unital classification theorem consists of unordered operator-algebraic $K$-theory, along with the space of all traces on $A$ and a pairing map which relates traces and $K$-theory. 
We set out a quick description of the components of the invariant to formalize $KT_u(\cdot)$ in first-order logic; those interested in a more thorough introduction of it, see \cite[\S 2.1]{CGSTW-new-classificationI-2023}.

Operator-algebraic $K$-theory consists of a functor $K_* = (K_0, K_1)$ from $C^*$-algebras to pairs of abelian groups. 
An element $p \in A$ is called a projection if $p^2=p^*=p$. 
Let $P_{n}(A)$ be the set of projections in $M_n(A)$ (which can be equipped with a unique $C^*$-algebra structure), and let $P_{\infty}(A)$ be the disjoint union of $P_n(A)$'s. 
Consider the Murray-von Neumann equivalence relation $\sim_0$ on $P_{\infty}(A)$ given by $p \sim_0 q \iff p = v^*v \wedge q=vv^*$ for some $v \in M_{m, n}(A)$, if $p \in M_n(A), q \in M_m(A)$.
We also define 
$p \oplus q = \begin{pmatrix}
    p \ 0 \\
    0 \ q
\end{pmatrix}$ for $p,q \in P_{\infty}(A)$. 
Then $P_{\infty}(A)/\sim_0$ becomes an abelian semigroup with $\oplus$, and its Grothendieck group is denoted by $K_0(A)$.
$K_1(A)$ is an abelian group that can be defined as $K_0(SA)$, where $SA := C_0((0,1),A)$ (there is a unique $C^*$-algebra structure on $SA$, although the definition of $K_0$ of a non-unital $C^*$-algebra differs slightly from the one we described).
The space of traces $T(A)$ is a compact convex set of all tracial states on $A$; a tracial state is a unital positive\footnote{i.e., taking elements of the form $b^*b$ to non-negative numbers} linear functional $\tau \colon A \rightarrow \mathbb{C}$ which additionally satisfies $\tau(ab) = \tau(ba)$ for all $a,b \in A$. 
We equip $T(A)$ with the weak-$*$ topology and denote by $\Aff T(A)$ the space of affine continuous real valued functions on $T(A)$.
Finally, we define a group homomorphism $\rho_A \colon K_0(A) \rightarrow \Aff T(A)$ 
specified by
\[
    \rho_A([p]_0 - [q]_0) (\tau) := \tau_n(p-q),
\]
where $\tau_n := \tau \otimes \operatorname{Tr}_n$ is the induced (non-normalized) trace on $M_n(A) \cong A \otimes M_n(\mathbb{C})$ and $[p]_0$ is the equivalence class of $p$ with respect to the Murray-von Neumann equivalence.

The pairing map induces an order on $K_0(A)$ where the positive cone of $K_0(A)$ is defined by:
\[
    x \geq 0 \iff \rho_A(x) (\tau) \geq 0 \, \text{ for all } \tau \in T(A).
\]

\begin{definition}{\cite[Definition 2.3]{CGSTW-new-classificationI-2023}}
    $KT_u$ is the functor on unital separable $C^*$-algebras which assigns to $A$ the quadruple
    \[
        KT_u(A) = (K_*(A), [1_A]_0, (\Aff T(A), (\Aff T(A))_+, u), \,  \rho_A),
    \]
    where $(\Aff T(A))_+$ is the positive cone of affine functionals over $T(A)$ and $u=1$ is the order unit of $\Aff T(A)$.
\end{definition}

The target space for the functor $KT_u$ consists of tuples $((G_0, G_1), c, (X, X_+, u), \rho)$, where $G_0, G_1$ are abelian groups, $c \in G_0$ is a distinguished element, $X$ is an ordered real vector space with an order unit $u$ (see Example \ref{example:language Lous}), and $\rho \colon G_0 \rightarrow X$ is a linear map taking $c$ to $u$.
\begin{definition}
    We define a metric first-order language $\Lang_{\inv}$ that consists of:
\begin{itemize}
    \item 2 sorts $G_0, G_1$ and a family of sorts $B_n$ for $n \in \mathbb{N}$;
    \item For the sorts $G_0$ and $G_1$ the language $L_{ab} = \{ +, 0 \}$ of abelian groups with identity;
    \item For the sort $G_0$, a constant symbol $c$;
    \item the language $L_{ous}$ for the sorts $B_n$;
    \item A family of predicate symbols $\rho_n \colon G_0 \rightarrow B_n$.
\end{itemize}
\end{definition}
For a unital $C^*$-algebra $A$, we can interpret $KT_u(A)$ as an $\Lang_{\inv}$-structure by interpreting $G_0$ as $K_0(A)$ with the distinguished element $c$ interpreted as $[1_A]_0$, $G_1$ as $K_1(A)$, $B_n$ as the radius $n$ ball in $\Aff T(A)$ with $P_n$ being the distance from the positive cone (see Example \ref{example:language Lous}), and $\rho_n$ as a truncated pairing map given by the composition of $\rho_A \colon K_0(A) \to \Aff T(A)$ with a retraction of $\Aff T(A)$ onto its radius $n$ ball (which sends $x$ to $nx/\max(\|x\|,n)$).
Sorts $G_0, G_1$ are equipped with the discrete metric, and $B_n$'s are equipped with the metric coming from the supremum norm on $\Aff T(A)$. We leave the determination of moduli of continuity of symbols in $\Lang_{\inv}$ to the reader.

Recall from Section \ref{sec: infinitary metric logic} that $\Mod_\omega(\Lang_{\inv})$ denotes the Borel parametrization of all separable $\Lang_{\inv}$-structures.
In order to apply the results of Section \ref{subsection: barwise}, we must verify that the invariant $KT_u$ can be lifted to a Borel function.
Note that for our main application we do not need to know Borelness of the functor between the underlying Borel categories. We use the following.
\begin{theorem}[Theorem 3.3, \cite{farah2013descriptive}]
    There is a Borel space structure on the space of Elliott invariants $\mathbf{Ell}$ such that the map $Ell \colon \Gamma_u \rightarrow \mathbf{Ell}$ which sends the encoding $\gamma$ of a unital separable $C^*$-algebra to $Ell(C^*(\gamma))$ is Borel.
\end{theorem}
The Elliott invariant is historically used as the invariant in the classification theorem; it is simply $KT_u$, with the extra information of an order on the $K_0$ group.

\begin{corollary} \label{corollary: KT_u is a Borel invariant}
    There is a Borel function $KT_u \colon \Mod_{\omega}(T_{C^*}') \to \Mod_{\omega}(\Lang_{\inv})$ such that $KT_u(A)$ encodes the $KT_u$-invariant of the $C^*$-algebra $\widehat{A}$.
\end{corollary}

Since this is only a translation of results from \cite{farah2013descriptive, Farahturbulence2014} to our setting, in the following proof we freely use notations and definitions from loc.cit.

\begin{proof}
    By Lemma \ref{lemma:our presentation of C* is equivalent to the standard one} and \cite[Lemma 3.15]{Farahturbulence2014}, it is enough to prove that there is a Borel map $\mathbf{Ell} \rightarrow \Mod_\omega(\Lang_{\inv})$ that sends
    \[
        (G_0,G_1,T,r) \in \mathbf{\Ell} \subseteq \mathbf{G}'_{\operatorname{ord}} \times \mathbf{G}_a \times \Kconv \times \bigsqcup_{n \in \mathbb{N} \cup \{ \mathbb{N} \}} C(\Delta^{\mathbb{N}}, \Delta^n),
    \]
    to some $A \in \Mod_\omega(\Lang_{\inv})$ that represents the same invariant. 
    It is clear that there is a Borel function $\alpha \colon \mathbf{G}'_{\operatorname{ord}} \to \Mod_{\omega}(L_{ab})$, where $L_{ab}' = \{0, +, c\}$ is the language of abelian groups with a constant, such that $G_0$ is sent to a code of the corresponding abelian group with a constant. 
    Similarly, there is a Borel isomorphism $\beta \colon \mathbf{G}_a \to \Mod_{\omega}(L_{ab})$ of these Borel parametrizations of the category of countable abelian groups.

    Now we argue that there is a Borel map $\delta \colon \Kconv \to \Mod_{\omega}(L_{ous})$ sending $T$ to a code for the separable $L_{ous}$-structure coming from the Archimedean order unit space $\Aff(T)$. 
    By \cite[Lemma 4.4]{Farahturbulence2014} there is a Borel map $\Psi \colon \Kconv \to S\bigl(C(\Delta, \mathbb{R})\bigr)$ that maps a compact convex set $K \subseteq \Delta^{\mathbb{N}}$ to a closed subspace $\Psi(K)$ of the Banach algebra $C(\Delta, \mathbb{R})$ (equipped with the supremum norm), that is isomorphic to $\Aff(K)$. 
    By the Kuratowski–Ryll-Nardzewski measurable selection theorem, there are Borel maps
    \[ f_n \colon S\bigl(C(\Delta, \mathbb{R})\bigr) \to C(\Delta, \mathbb{R}), \]
    such that for every $X \in S\bigl(C(\Delta, \mathbb{R})\bigr)$ (i.e., for every closed subspace $X \subseteq C(\Delta, \mathbb{R})$), the set $\{ f_n(X) \colon n \in \mathbb{N} \}$ is dense in $X$.
    Thus 
    \[ \{ F(f_1(\Psi(K)), \dots, f_k(\Psi(K))) \colon F \in \mathbb{Q}[x_1, x_2, \dots]_{\deg \leq 1} , k \in \mathbb{N} \}\] is a countable dense $\mathbb{Q}$-subspace of $\Psi(K)$. 
    By choosing a bijection between $\mathbb{Q}[x_1, x_2, \dots]_{\deg \leq 1}$ and $\omega$, we get the desired map $\delta \colon \Kconv \to \Mod_{\omega}(L_{ous})$.
    
    The last thing to check is that computing the pairing $\rho$ from a tuple $(G_0,G_1,T,r)$ is Borel. 
    Recall that the map $\Psi \colon \Kconv \to S\bigl(C(\Delta, \mathbb{R})\bigr)$ from \cite[Lemma 4.4]{Farahturbulence2014} is constructed as follows:
    \begin{itemize}
        \item a continuous map $\chi' \colon \Kconv \to C(\Delta^{\mathbb{N}}, \Delta^{\mathbb{N}})$ is constructed (note that in \cite[Lemma 3.1]{farah2013descriptive} this map is called $\Psi$), such that $\chi'(K)$ is a continuous retraction onto $K$;
        \item a continuous surjective map $\eta \colon \Delta \to \Delta^{\mathbb{N}}$ is chosen;
        \item $\chi \colon \Kconv \to C(\Delta, \Delta^{\mathbb{N}})$ is defined by $\chi(K) = \chi'(K) \circ \eta$;
        \item one puts
        \[ \mathcal{Y} := \{ (K,f) \in \Kconv \times C(\Delta, \mathbb{R}) \colon f = g \circ \chi(K) \textnormal{ for some } g \in \Aff(K) \} \]
        and defines $\Psi(K) := \{ f \colon (K,f) \in \mathcal{Y} \}$.
    \end{itemize}
    By \cite[Definition 3.2]{farah2013descriptive}, $r \in \Pairing(T,G_0)$, which means that $r = r' \circ \chi'(T)$, for some $r' \colon T \to \States(G_0) \subseteq \mathbb{R}^{n}$ (where $n$ corresponds to the cardinality of the group encoded by $G_0$). We define $\rho = \rho_{(G_0,G_1,T,r)} \colon n \to \Psi(T) \subseteq C(\Delta, \mathbb{R})$ by
    \[ \rho(g)  = \proj_g \circ \ r \circ \eta, \]
    where $\proj_g \colon \mathbb{R}^n \to \mathbb{R}$ is given by the projection onto $g$'th factor (for $g \in n$ that we think of as an element of the group encoded by $G_0$). 
    Note that the functions
    \[\begin{tikzcd}
    	{C(\Delta^{\mathbb{N}}, \Delta^{n})} & {C(\Delta^{\mathbb{N}}, \mathbb{R})} & {C(\Delta, \mathbb{R})}
    	\arrow["{\proj_g \circ (-)}", from=1-1, to=1-2]
    	\arrow["{(-) \circ \eta}", from=1-2, to=1-3]
    \end{tikzcd}\]
    are continuous, hence, for $g \in n$, the partial map $\mu \colon \Ell \to C(\Delta, \mathbb{R})$ sending $(G_0,G_1,T,r)$ to $\rho(g)$ is Borel.
    Thus, for each $F \in \mathbb{Q}[x_1, x_2, \dots]_{\deg \leq 1}$ and $g \in n$ the map $\nu \colon \mathbf{Ell} \to \mathbb{R}$ taking $(G_0,G_1,T,r)$ to the distance between $\rho(g)$ and $F(f_1(\Psi(T)), \dots, f_k(\Psi(T)))$, in the Banach space $C(\Delta, \mathbb{R})$, is Borel. Using maps $\alpha, \beta, \delta$ and $\nu$'s, one can then produce $KT_u \colon \Mod_{\omega}(T_{C^*}') \to \Mod_{\omega}(\Lang_{\inv})$ with the desired property.
\end{proof}

\subsection{Borel properties of separable $C^*$-algebras}\label{section:Borel_properties_of_C_star_alg}

Let $\Cstar = \bfMod_\omega(T_{C^*})$ denote the Borel category of $C^*$-algebras. Since the theory $T_{\textnormal{$C^*$}}$ from Example \ref{example:language of C*} forces the metric to be determined by the norm $\|-\|$ and sorts to correspond to balls with respect to $\|-\|$, we will assume that $\Cstar_{\Ob}$ is the set of tuples
\[ (+,\times,(-)^{*}, \|-\|, (\lambda \cdot)_{\lambda \in \mathbb{Q}(i)}) \in  \omega^{\omega^2} \times \omega^{\omega^2} \times \omega^{\omega} \times \mathbb{R}_{\geq 0}^{\omega} \times (\omega^{\omega})^{\mathbb{Q}(i)}\]
that induce a $C^*$-algebra structure on the completion of $\omega$ with respect to the metric $d(a,b) = \|a-b\|$. 
Similarly, we encode a morphism $\widehat{f} \colon \widehat{A} \rightarrow \widehat{B}$ between $C^*$-algebras corresponding to $A, B \in \Cstar_{\Ob}$, by a triple $(A, B, F)$, where $F(a,b) := d_B(\lim_n f_n(a), b)$ for some representative sequence $(f_n)$ converging to $\widehat{f}$ (in the sense of Lemma \ref{lemma: Cauchy sequences X_1 Borel}).
In other words, we identify $\Cstar_{\Ob}$ with $\hat{\Xi}$ introduced in Section~\ref{subsec: borel parametrization of C*} (using Lemma \ref{lemma:our presentation of C* is equivalent to the standard one}) and we use conventions from Section \ref{subsection: Borel categories}. For applications, we need to know that certain classes of morphisms in $\Cstar_{\Mor}$ are Borel.
\begin{lemma} \label{lem: injective Borel}
    The set of injective morphisms in $\Cstar_{\Mor}$ is Borel. 
\end{lemma}
\begin{proof}
    Let $I$ be the collection of all $(A, B, F) \in \Cstar_{\Mor}$ such that the induced morphism $\widehat{f} \colon \widehat{A} \to \widehat{B}$ is injective.
    Recall that $*$-homomorphisms between $C^*$-algebras are contractive. As $\widehat{f}$ is injective, and images of $C^*$-algebras under $*$-homomorphisms are closed, we may consider the $*$-homomorphism $\widehat{f}^{-1} \colon \widehat{f}(\widehat{A}) \rightarrow \widehat{A}$. Thus both $\widehat{f}$ and $\widehat{f}^{-1}$ are contractive, i.e., $\widehat{f}$ is an isometry. Then the set $I$ is captured by the Borel condition
    \[  
        \bigwedge_{a \in A} F(a,0) = \norm{a}.   \qedhere
    \]
\end{proof}

\begin{lemma}
    The set of surjective morphisms in $\Cstar_{\Mor}$ is Borel.
\end{lemma}
\begin{proof}
    Let $I'$ be the set of all $(A, B, F) \in \Cstar_{\Mor}$ such that $\widehat{f} \colon \widehat{A} \to \widehat{B}$ is surjective.
    The set $I'$ can be defined by the condition ``for all elements $b$ in $B$, there exists some $a$ in $A$ such that $F$ maps $a$ close to $b$."
    Formally, this is captured by the Borel condition:
    \[ \bigwedge_{b \in B} \bigwedge_{\varepsilon >0} \bigvee_{a \in A} F(a,b) < \varepsilon.    \qedhere 
    \]
\end{proof}

\begin{corollary}
    The set of isomorphisms in $\Cstar_{\Mor}$ is Borel.
\end{corollary}

\begin{lemma}\label{lemma:split_SES_are_Borel}
    Let $J$ be the set of triples $(f, g, h) \in \Cstar_{\Mor}$ such that $(f,g,h)$ form a split short exact sequence, i.e.,
    \begin{itemize}
        \item $\widehat{f}$ is injective,
        \item $\widehat{g} \circ \widehat{h} = \id_C$,
        \item $\kernel \widehat{g} = \image \widehat{f}$.
    \end{itemize}
    Then $J$ is Borel.
\end{lemma}
\begin{proof}
    The first bullet is Borel by Lemma \ref{lem: injective Borel}. The second bullet is Borel because the composition map is Borel (see Proposition \ref{proposition:Borel_category_of_L_structures}), along with the fact that $\id_C = [i(C)]$ is a constant symbol in our language.

    To encode exactness in the middle, recall from Proposition \ref{proposition:Borel_category_of_L_structures} that elements in $\Cstar_{\Mor}$ are encoded by elements $f = (A, B, F)$ and $g = (B, C, G)$, where $F, G$ encode graphs of functions. Thus, $\kernel \widehat{g}$ is encoded by the dense set $\{b \in B \colon G(b,0) = 0\}$ and similarly $\image \widehat{f}$ is encoded by $\{ b \in B \colon \exists a \in A, \, F(a,b) = 0\}$. 
    Thus, the set in question is Borel since it can be written as:
    \[
        \bigwedge_{b \in B} [G(b,0) = 0 \iff \bigwedge_{\varepsilon > 0} \bigvee_{a \in A} F(a, b) < \varepsilon]. \qedhere
    \]
\end{proof}

Having shown that the collection of split exact sequences is Borel, we move on to the stability property, namely that $KK$ is unaffected by tensoring with the compact operators $\mathcal{K}$.

\begin{construction}\label{construction:matrix_amplifications}
    Let $A \in \Cstar_{\Ob}$ and consider the (directed) union of $n \times n$ matrices with coefficients in $A$, for all natural $n$. Denote this set by $M_0(A)$. Transport $+,\times, (-)^*, (\lambda \cdot)_{\lambda \in \mathbb{Q}(i)}$ from $A$ to $M_0(A)$, where the operations are naturally interpreted in the matrix algebra. Moreover, define
    \[ \| (a_{ij}) \| := \sup \bigg\| \sum_{i=1}^n \sum_{j=1}^n c_i a_{ij} d_j^* \bigg\|, \]
    where the supremum is taken over tuples $c_1, \dots, c_n \in A, d_1, \dots, d_n \in A$ such that
    \[ \bigg\| \sum_{i=1}^n c_i c_i^* \bigg\| \leq 1 \textnormal{ and } \bigg\| \sum_{i=1}^n d_i d_i^* \bigg\| \leq 1. \]
    Fix a (computable) bijection between $M_0(\omega)$ and $\omega$ (recalling that $A$ is indexed by $\omega$). 
    Using this bijection, one can see $M_0(A)$ as an element of $\Cstar_{\Ob}$.
\end{construction}

\begin{lemma}\label{lemma:matrix_amplifications_description}
    Under the notation of the above construction, $M_0(A) \in \Cstar_{\Ob}$ and $\widehat{M_0(A)}$ is isomorphic to $\widehat{A} \otimes \fancyK$.
\end{lemma}
\begin{proof}
    The norm defined in Construction \ref{construction:matrix_amplifications} is given by the natural action of $M_n(\widehat{A})$ on the Hilbert $C^*$-module $\widehat{A}^n$, with inner product given by $\ip{(a_i), (b_i)} = \sum_{i=1}^n a_i^* b_i$. 
    One can use this characterization to check that the norm satisfies the $C^*$-identity. 
    Now, recall that taking tensor products and direct limits of $C^*$-algebras commute whenever the connecting maps of the directed limit are faithful. Since we consider the map from $M_n(A)$ to $M_{n+1}(A)$ given by $a \mapsto \begin{pmatrix} a & 0 \\ 0 & 0 \end{pmatrix}$, this is satisfied. It follows that the $C^*$-algebra given by $\varinjlim M_n(\widehat{A})$ is isomorphic to $\varinjlim M_n(\mathbb{C}) \otimes \widehat{A} = \mathcal{K} \otimes \widehat{A}$. It remains to argue that $\varinjlim M_n(\widehat{A})$ is isomorphic to the norm-closure of $\varinjlim M_n(A)$. This follows once we observe a relationship between the norm on $A$ and the norm on $M_n(A)$. For all $(a_{ij}) \in M_n(A)$ and tuples $c_1, \dots, c_n \in A, d_1, \dots, d_n \in A$, 
    \begin{align*}
        \bigg\| \sum_i c_i \sum_j a_{ij} d_j^* \bigg\| &\leq \left(\sum_i \norm{c_i}^2 \right)^{1/2}  \left( \sum_i \bigg\| \sum_j a_{ij} d_j^* \bigg\| ^2 \right) ^{1/2} \\
        &\leq \left(\sum_i \norm{c_i}^2 \right)^{1/2} \left( \sum_{ij} \norm{a_{ij}^2} \sum_j \norm{d_j^*}^2 \right)^{1/2} \\
        &\leq \left(\sum_i \norm{c_i^* c_i} \right)^{1/2} \left( \sum_{ij} \norm{a_{ij}^2} \right)^{1/2} \left( \sum_j \norm{d_j d_j^*} \right)^{1/2} \\
        &\leq \left( \sum_{ij} \norm{a_{ij}^2} \right)^{1/2}.
    \end{align*} 
    Thus, an $n^2$-tuple of Cauchy sequences in $A$ always induces a Cauchy sequence in $M_n(A)$.
\end{proof}

\begin{lemma} \label{lemma:stabilization is Borel}
    There is a function $M_1 \colon \Cstar_{\Ob} \to \Cstar_{\Ob}$, such that $M=(M_0,M_1)$ is a Borel functor that is (naturally isomorphic to) the functor on the category of separable $C^*$-algebras which sends at $C^*$-algebra $\widehat{A}$ to its stabilization $\widehat{A} \otimes \mathcal{K}$. 
    
    Moreover, there is a Borel natural transformation $\alpha \colon \id \to M$ taking $A$ to the map $\alpha_A$ from $A$ to $M_0(A)$ embedding $A$ into the top-left corner.
\end{lemma}
\begin{proof}
    It is clear from the definition that $M_0 \colon \Cstar_{\Ob} \to \Cstar_{\Ob}$ is Borel. We construct $M_1 \colon \Cstar_{\Mor} \to \Cstar_{\Mor}$ using Lemma~\ref{lemma:matrix_amplifications_description}. More precisely, for $\widehat{f} \colon \widehat{A} \to \widehat{B}$, we look at $\widehat{f} \otimes \id_{\fancyK} \colon \widehat{A} \otimes \fancyK \to \widehat{B} \otimes \fancyK$ and define $M_1(f)=(M_0(A),M_0(B),H) \in \Cstar_{\Mor}$ by taking $H$ to be induced by the *-homomorphism $\widehat{h} \colon \widehat{M_0(A)} \to \widehat{M_0(B)}$ obtained from $\widehat{f} \otimes \id_{\fancyK}$. 
    
    It follows that $M=(M_0,M_1)$ is (naturally isomorphic to) the stabilization-by-$\mathcal{K}$ functor, and we are left with showing that $M_1$ is Borel. Using Suslin's theorem, it is enough to see that the graph of $M_1$ is analytic. Let $f:=(A,B,(f_n)_{n \in \omega})$ be an element inducing $\widehat{f} \colon \widehat{A} \to \widehat{B}$ (we use the notation from Lemma \ref{lemma: Cauchy sequences X_1 Borel} and Proposition \ref{proposition:Borel_category_of_L_structures}). We define $g:=(M_0(A),M_0(B), (g_n)_{n \in \omega})$ by letting $g_n$ to be the map $M_0(A) \to M_0(B)$ induced by $f_n$ coordinate-wise. We claim that $\widehat{g} = \widehat{h}$. Indeed, the key fact is that the norm of $(a_{ij}) \in M_n(A)$ is bounded above by the norms of the matrix coefficients $a_{ij}$ in $A$, as shown in the proof of Lemma \ref{lemma:matrix_amplifications_description}. This implies that $(g_n)_{n \in \omega}$ is Cauchy in a way that depends explicitly on $(f_n)_{n \in \omega}$, ensuring compatibility of the limits $\widehat{g}, \widehat{h}$. This shows that the graph of $M_1$ is analytic, by presenting it as a Borel image of the set $X_1$ from Lemma \ref{lemma: Cauchy sequences X_1 Borel} (in the sense of the convention we took at the beginning of Section \ref{section:Borel_properties_of_C_star_alg}).
    We leave the proof of the moreover part to the reader.
\end{proof}

\begin{corollary} 
    The set of morphisms in $\Cstar_{\Mor}$ that are of the form $\alpha_A$ for $A \in \Cstar_{\Ob}$ is analytic.
\end{corollary}

Lastly, we address analyticity of the homotopy relation.
Recall that two $*$-homomorphisms $\widehat{f}, \widehat{g} \colon \widehat{A} \rightarrow \widehat{B}$ are homotopic if there is a $*$-homomorphism $\widehat{h} \colon \widehat{A} \to C([0,1], \widehat{B})$ such that $\ev_0 \circ \widehat{H} = \widehat{f}$ and $\ev_1 \circ \widehat{F} = \widehat{f}'$. Thus, we first discuss Borelness of the construction which takes $B \in \Cstar_{\Ob}$ to an encoding of the $C^*$-algebra $C([0,1], \widehat{B})$. 

\begin{construction} \label{constr: tensor with C[0,1]}
    Let $B \in \Cstar_{\Ob}$ be the encoding of the $C^*$-algebra $\widehat{B}$ and consider the algebra $C([0,1], \widehat{B})$ of continuous functions from $[0,1]$ to $\widehat{B}$ with respect to the norm topology. 
    Denote by $G_0(B)$ the following set:
    \[ \left\{ \sum_i f_{i}(t) b_{i} \colon f_{i}(t) \in \mathbb{Q}(i)[t], b_{i} \in B \right\}, \]
    where we take finite sums. Define addition, multiplication, multiplication by $\mathbb{Q}(i)$-scalars and $(-)^*$ in the natural way. Also, define a function $\|-\|$ by
    \[ \Big\| \sum_i f_{i}(t) a_{i} \Big\| := \sup_n \max_{j=0, \dots, n} \Big\| \sum_i  f_{i} \left(\frac{j}{n} \right) b_{i} \Big\|_B. \]
    Then $G_0(B) \in \Cstar_{\Ob}$ and the completion of $G_0(B)$ is isomorphic to $C([0,1], \widehat{B})$.
\end{construction}

Indeed, that $\widehat{G_0(B)} \cong C([0,1], \widehat{B})$ as $C^*$-algebras follows since the inclusion map $G_0(B) \rightarrow C([0,1], \widehat{B})$ is isometric, and the latter object is norm-closed. 

\begin{lemma}\label{lemma:continuous_functions_from_01_to_x_Borel_functor}
    There is a function $G_1 \colon \Cstar_{\Mor} \to \Cstar_{\Mor}$, such that $G=(G_0,G_1)$ is a Borel endofunctor of $\Cstar$ that is (naturally isomorphic to) the functor $B \mapsto C([0,1], B)$ on the category of separable $C^*$-algebras. Moreover, evaluations $\ev_t$ (for $t \in [0,1]$) are Borel natural transformations from $G$ to the identity functor. 
\end{lemma}
\begin{proof}
    The proof is similar to that of Lemma \ref{lemma:stabilization is Borel}. In particular, for any $*$-homomorphism $\widehat{f} \colon \widehat{A} \rightarrow \widehat{Y}$ we take $G_1 \colon \Cstar_{\Mor} \rightarrow \Cstar_{\Mor}$ to be induced by the $*$-homomorphism $\id_{C[0,1]} \otimes \widehat{f} \colon C([0,1], \widehat{A}) \rightarrow C([0,1], \widehat{B})$.  The other arguments are similar, with analyticity of $G_1$ following due to the relation between the norms:
    \[
        \left \lVert \sum_i f_i(t) a_i \right \rVert_{G_0(A)}  = \sup_n \max_{j = 0, \dots, n} \left \lVert \sum_i f_i \left( \frac{j}{n} \right) a_i \right \rVert_A \leq M \sum_i \norm{a_i}_{A},
    \]
    where $M$ is the maximum absolute value of the polynomials $f_1, \dots, f_n$ on $[0,1]$.
    
\end{proof}

The above lemma is related to \cite[Lemma 3.10]{Farahturbulence2014}, where the Borelness of $G_0$ (and of the functor of tensoring with any separable nuclear $C^*$-algebra) is shown. We conclude with the analyticity of the set of pairs of homotopic morphisms.

\begin{lemma}\label{lemma:pairs_of_homotopic_maps}
    Let $H \subset (\Cstar_{\Mor})^2$ be the set of pairs of $f, f'$ such that their sources and targets are the same respectively, and $\widehat{f}$ is homotopic to $\widehat{f}'$. Then $H$ is analytic.
\end{lemma}
\begin{proof}
    This is by Lemma~\ref{lemma:continuous_functions_from_01_to_x_Borel_functor}, because we can write
    \[ H = \{ (\ev_0(F) , \ev_1(F)) \colon (\exists F \in \Cstar_{\Mor})(\exists B \in \Cstar_{\Ob}) (t(F) = G_0(B)) \}, \]
    where $t$ is the target map in the Borel category $\Cstar$.
\end{proof}

\subsection{Borel complexity of $KK$-equivalence}\label{section: analycity of KK}

We have now established all the groundwork to define our theory $T_{KK}$ for discussing $KK$ as a universal functor. We use the notation and terminology from Section \ref{subsection: analytic conditions}, so in the following definition we abuse the notation by writing the $\circ$ symbol, instead of appropriate $c_{A,B,C}$'s. We also define a language $L_{KK}$ as $L_0$ from Construction \ref{construction: describing L_0 depending on X} with $\mathbf{X}$ being $\Cstar$.
\begin{definition} \label{definition: T_KK}
    Let $T_{KK}$ be the $L_{KK}$-theory consisting of $T_0$, as well as:
    \begin{itemize}
        \item Abelian group axioms on $(\Mor(A,B),+_{A,B},0_{A,B})$ for all $A, B \in \Cstar_{\Ob}$;
        \item A family of axioms 
        \[\textnormal{HI}_{f,f'} \colon~[f]=[f'],\]
        for each $A, B\in \Cstar_{\Ob}$ and $f,f'\in \Cstar_{\Mor}(A,B)$ with $\widehat{f}$ homotopic to $\widehat{f}'$;
        \item A family of axioms 
        \[\textnormal{Stab}_{A} \colon~ \exists \alpha\in \Cstar_{\Mor}(M_0(A),A)~\big(\alpha\circ g=[i(A)]\wedge g\circ\alpha= [i(M_0(A))]\big)\]
        for each $A \in \Cstar_{\Ob}$ and $g\in \Cstar_{\Mor}(A,M_0(A))$ corresponding to the natural inclusion $\widehat{A} \to \widehat{A}\otimes\mathcal{K}$;
        \item A family of axioms
        \begin{align*}
            \textnormal{SES}_{f,g,h}^{D} \colon~\forall\alpha,\alpha'\in\Mor(D, A) 
        \big([f]\circ\alpha=[f]\circ\alpha'\rightarrow \alpha=\alpha'\big) \\
        ~\wedge~\forall\beta\in\Mor(D, C)\big([g\circ h]\circ\beta=\beta\big)
        \end{align*}
        for all split short exact sequences
        \[
        \begin{tikzcd}
	0 & A & B & C & 0
	\arrow[from=1-1, to=1-2]
	\arrow["f", from=1-2, to=1-3]
	\arrow["g", curve={height=-6pt}, from=1-3, to=1-4]
	\arrow["h"', curve={height=-8pt}, from=1-4, to=1-3]
	\arrow[from=1-4, to=1-5]
        \end{tikzcd}
        \]
        and all $D \in \Cstar_{\Ob}$.
        
    \end{itemize}
\end{definition}

\begin{proposition}
    $T_{KK}$ is an analytic condition.
\end{proposition}
\begin{proof}
    Indeed, the abelian group axioms can easily seen to be Borel.
    The homotopy axioms $\text{HI}_{f, f'}$ are all analytic by Lemma \ref{lemma:pairs_of_homotopic_maps}.
    The stability axioms $\text{Stab}_{A}$ are Borel by Lemma \ref{lemma:stabilization is Borel}, and the split exact sequence axioms $\text{SES}_{f,g,h}^{D}$ are Borel by Lemma \ref{lemma:split_SES_are_Borel}.
\end{proof}

\begin{remark} \label{remark:models are functors description}
    Recall from Proposition \ref{proposition: Borel model of the category of sigma} that $\Cstar$ is a Borel model of the category of separable $C^*$-algebras, or more precisely the completion functor $G := (\widehat{-}) \colon \Cstar \to \mathcal{C}^*$ is an equivalence. 
    Choose an inverse $H \colon \mathcal{C}^* \to \Cstar$ of $G$.

    By Construction \ref{construction: describing L_0 depending on X}, a model $M \models T_{KK}$ defines a functor $F' \colon \Cstar \to \mathcal{A}$, for some abelian category $\mathcal{A}$. 
    By composing with $H$ we get a functor $F^M := F' \circ H \colon \mathcal{C}^* \to \mathcal{A}$, and axioms of $T_{KK}$ ensure that $F$ is a $\kappa\kappa$-functor. 
    
    On the other hand, given a $\kappa\kappa$-functor $F \colon \mathcal{C}^* \to \mathcal{A}$, for some abelian category $\mathcal{A}$, we can construct a model of $M \models T_{KK}$ by first composing with $G$ to get $F \circ G \colon \Cstar \to \mathcal{A}$ and then putting $\Mor(A,B)^M := \Mor_{\mathcal{A}}(F(G(A)),F(G(B)))$ (similarly we define interpretations of constants, etc.). 
\end{remark}

\begin{theorem}
    The set of pairs $(A, B)\in (\Cstar_{\Ob})^2$ such that $\widehat{A}$ and $\widehat{B}$ are $KK$-equivalent is analytic.
\end{theorem}
\begin{proof}
    By Remark \ref{remark:models are functors description}, $\widehat{A}$ and $\widehat{B}$ are $KK$-equivalent if and only if for every model $M \models T_{KK}$, they become equivalent after applying the functor corresponding to $M$. 
    Thus
    \[ T_{KK} \models (\exists \alpha \in \Mor(A, B))(\exists \beta \in \Mor(B, A))( \alpha \circ \beta = [i(B)] \wedge \beta \circ \alpha = [i(A)] ), \]
    and analyticity follows from Theorem \ref{thm:T_equivalence_is_analytic}. 
\end{proof}

Combining with Fact \ref{fact: UCT KK-equiv-to-commutative} and the fact that commutative $C^*$-algebras are an axiomatizable class allows us to conclude.

\begin{theorem}
    The set of unital separable $C^*$-algebras satisfying the UCT is analytic.
\end{theorem}

With this in hand, Theorem \ref{main1: classifiable} follows as described at the end of the introduction: we apply the abstract infinitary classification theorem (Corollary \ref{cor: infinitary classification}) using $X \subseteq \Mod_{\omega}(L_{C^*})$ as the set of classifiable $C^*$-algebras and $Y \subseteq \Mod_{\omega}(L_{\inv})$ as the set of possible invariants. 
The unital classification theorem (Theorem~\ref{thm: classification of C*-algebras}), the fact that the other assumptions of classification are Borel (Theorem \ref{thm: USSN Z-stable is Borel}), and the fact that the $KT_u$ map is Borel (Corollary \ref{corollary: KT_u is a Borel invariant}) allow us to satisfy the assumptions of Corollary \ref{cor: infinitary classification}.

\subsection{Questions}
Since the other assumptions of the unital classification theorem are Borel, this begs the natural question.
\begin{question}
    Is the UCT for unital $C^*$-algebras strictly analytic?
\end{question}

If the UCT restricted to the Borel set of nuclear $C^*$-algebras results in a  strictly analytic set, then this would answer the UCT problem in the negative, as the set of nuclear $C^*$-algebras would differ from the set of nuclear $C^*$-algebras with the UCT.
A relaxation of the above question would be to ask the same for the $KK$-equivalence relation (see also \cite[Problem 9.3]{Farahturbulence2014}).
\begin{question}
    Is $KK$-equivalence strictly analytic? What is its Borel cardinality?
\end{question}
The following is also natural to ask, and is related to Problems 9.4, 9.7 in \cite{Farahturbulence2014}.
\begin{question}
    Can the $KT_u$-invariant be lifted to a Borel functor \[ KT_u \colon \bfMod_{\omega}(T_{C^*}') \to \bfMod_{\omega}(\Lang_{\inv})?\]
\end{question}
Note that in Corollary \ref{corollary: KT_u is a Borel invariant} we only prove Borelness on the level of objects. Borelness of the functor would imply that it comes from an infinitary interpretation, see \cite{MR3893282, MR4029718} (one would need an appropriate continuous logic generalisation). 

\printbibliography 

@article{games_on_AF_algebras,
  title={Games on AF-algebras},
  author={De Bondt, Ben and Vaccaro, Andrea and Veli{\v{c}}kovi{\'c}, Boban and Vignati, Alessandro},
  journal={International Mathematics Research Notices},
  volume={2023},
  number={23},
  pages={19996--20038},
  year={2023},
  publisher={Oxford University Press}
}

@article{EGLN-classification,
  title={On the classification of simple amenable $C^*$-algebras with finite decomposition rank, II.},
  author={Elliott, George A and Gong, Guihua and Lin, Huaxin and Niu, Zhuang},
  journal={Journal of Noncommutative Geometry},
  volume={19},
  number={1},
  year={2025}
}

@article{GLN-classif-1,
  title={A classification of finite simple amenable Z-stable $C^*$-algebras, I: C*-algebras with generalized tracial rank one},
  author={Gong, G and Lin, Huaxin and Niu, Z},
  journal={CR Math. Rep. Acad. Sci. Canada},
  volume={42},
  number={3},
  year={2020}
}

@article{GLN-classif-2,
  title={A classification of finite simple amenable Z-stable $C^*$-algebras, II: C*-algebras with rational generalized tracial rank one},
  author={Gong, G and Lin, Huaxin and Niu, Z},
  journal={CR Math. Rep. Acad. Sci. Canada},
  volume={42},
  number={4},
  year={2020}
}

@article{TWW-quasidiagonality,
  title={Quasidiagonality of nuclear $C^*$-algebras},
  author={Tikuisis, Aaron and White, Stuart and Winter, Wilhelm},
  journal={Annals of Mathematics},
  volume={185},
  number={1},
  pages={229--284},
  year={2017},
  publisher={Department of Mathematics, Princeton University Princeton, New Jersey, USA}
}

@book {Goldbring2023book,
     TITLE = {Model theory of operator algebras},
    SERIES = {De Gruyter Series in Logic and its Applications},
    VOLUME = {11},
    EDITOR = {Goldbring, Isaac},
 PUBLISHER = {De Gruyter, Berlin},
      YEAR = {2023},
     PAGES = {ix+484},
      ISBN = {978-3-11-076821-3; 978-3-11-076828-2; 978-3-11-076833-6},
   MRCLASS = {03-02 (03Cxx 46-02 46Lxx 47Lxx)},
  MRNUMBER = {4654487},
}

@article{robert2025selfless,
  title={Selfless C$^*$-algebras},
  author={Robert, Leonel},
  journal={Advances in Mathematics},
  volume={478},
  pages={110409},
  year={2025},
  publisher={Elsevier}
}

@article{goldbring2022existentially,
  title={Existentially closed W$^*$-probability spaces},
  author={Goldbring, Isaac and Houdayer, Cyril},
  journal={Mathematische Zeitschrift},
  volume={301},
  number={4},
  pages={3787--3816},
  year={2022},
  publisher={Springer}
}

@article{ioana2024existential,
  title={Existential closedness and the structure of bimodules of II$_1$ factors},
  author={Ioana, Adrian and Tan, Hui},
  journal={Journal of Functional Analysis},
  volume={286},
  number={4},
  pages={110264},
  year={2024},
  publisher={Elsevier}
}

@article{farah2014model,
  title={Model theory of operator algebras {II}: model theory},
  author={Farah, Ilijas and Hart, Bradd and Sherman, David},
  journal={Israel Journal of Mathematics},
  volume={201},
  pages={477--505},
  year={2014},
  publisher={Springer}
}

@inproceedings{kirchberg1995classif,
  title={Exact C$^*$-algebras, tensor products, and the classification of purely infinite algebras},
  author={Kirchberg, Eberhard},
  booktitle={Proceedings of the International Congress of Mathematicians: August 3--11, 1994 Z{\"u}rich, Switzerland},
  pages={943--954},
  year={1995},
  organization={Springer}
}

@article{ElliottEtAl2013,
  title = {The isomorphism relation for separable C\*-algebras},
  author = {George A. Elliott and Ilijas Farah and Vern Paulsen and Christian Rosendal and Andrew S. Toms and Asger T\o{}rnquist},
  journal = {Mathematical Research Letters},
  volume = {20},
  number = {6},
  pages = {1071--1080},
  year = {2013},
  doi = {10.4310/MRL.2013.v20.n6.a6},
}

@article{murray1943rings,
  title={On rings of operators IV},
  author={Murray, Francis J. and von Neumann, John},
  journal={Annals of Mathematics},
  volume={44},
  number={4},
  pages={716--808},
  year={1943},
  publisher={JSTOR}
}

@article{glimm1960,
  title={On a Certain Class of Operator Algebras},
  author={Glimm, James G.},
  journal={Transactions of the American Mathematical Society},
  volume={95},
  number={2},
  pages={318--340},
  year={1960},
  publisher={JSTOR}
}

@article{connes1976classification,
  title={Classification of injective factors cases II$_1$, II$_\infty$, III$_\lambda$, $\lambda$$\neq$ 1},
  author={Connes, Alain},
  journal={Annals of Mathematics},
  volume={104},
  number={1},
  pages={73--115},
  year={1976},
  publisher={JSTOR}
}

@misc{CGSTW-new-classificationI-2023,
      title={Classifying $^*$-homomorphisms I: Unital simple nuclear $C^*$-algebras}, 
      author={Jos\'e R. Carri\'on and James Gabe and Christopher Schafhauser and Aaron Tikuisis and Stuart White},
      year={2023},
      eprint={2307.06480},
      archivePrefix={arXiv},
      primaryClass={math.OA},
}

@article{phillips2000classif,
  title={A classification theorem for nuclear purely infinite simple C$^{*}$-algebras},
  author={Phillips, N. Christopher},
  journal={Documenta Mathematica},
  volume={5},
  pages={49--114},
  year={2000}
}

@article{elliott1976AF,
  title={On the classification of inductive limits of sequences of semisimple finite-dimensional algebras},
  author={Elliott, George A.},
  journal={Journal of Algebra},
  volume={38},
  number={1},
  pages={29--44},
  year={1976},
  publisher={Elsevier}
}

@article{Kasparov1981-origKK,
  title={The operator $K$-functor and extensions of C$^*$-algebras},
  author={Kasparov, Gennadi G.},
  journal={Mathematics of the USSR-Izvestiya},
  volume={16},
  number={3},
  pages={513},
  year={1981},
  publisher={IOP Publishing}
}

@article {Higson-KKtheory,
    AUTHOR = {Higson, Nigel},
     TITLE = {A characterization of {$KK$}-theory},
   JOURNAL = {Pacific J. Math.},
  FJOURNAL = {Pacific Journal of Mathematics},
    VOLUME = {126},
      YEAR = {1987},
    NUMBER = {2},
     PAGES = {253--276},
      ISSN = {0030-8730,1945-5844},
   MRCLASS = {46L80 (18F25 19D25 58G12)},
  MRNUMBER = {869779},
MRREVIEWER = {Vern\ Paulsen},
}

@article {MR1680321,
    AUTHOR = {Jiang, Xinhui and Su, Hongbing},
     TITLE = {On a simple unital projectionless {$C^*$}-algebra},
   JOURNAL = {Amer. J. Math.},
  FJOURNAL = {American Journal of Mathematics},
    VOLUME = {121},
      YEAR = {1999},
    NUMBER = {2},
     PAGES = {359--413},
      ISSN = {0002-9327,1080-6377},
   MRCLASS = {46L35 (19K35 46L80)},
  MRNUMBER = {1680321},
MRREVIEWER = {Vicumpriya\ S.\ Perera},
}

@article {MR4029718,
    AUTHOR = {Chen, Ruiyuan},
     TITLE = {Borel functors, interpretations, and strong conceptual completeness for {$\mathcal{L}_{\omega_1\omega}$}},
   JOURNAL = {Trans. Amer. Math. Soc.},
  FJOURNAL = {Transactions of the American Mathematical Society},
    VOLUME = {372},
      YEAR = {2019},
    NUMBER = {12},
     PAGES = {8955--8983},
      ISSN = {0002-9947,1088-6850},
   MRCLASS = {03C75 (03C15 03E15 03G30 18C10)},
  MRNUMBER = {4029718},
MRREVIEWER = {Predrag\ Tanovi\'c},
       DOI = {10.1090/tran/7950},
}

@article{Kechrisdescriptivecstar,
  title={The descriptive classification of some classes of C$^*$-algebras},
  author={Kechris, Alexander S.},
  journal={Proceedings of the Sixth Asian Logic Conference (Beijing, 1996)},
  year={1998},
  pages={121--149},
  publisher={World Sci. Publ.}
}

@article{farah2013descriptive,
  title={The descriptive set theory of C$^*$-algebra invariants},
  author={Farah, Ilijas and Toms, Andrew and T{\"o}rnquist, Asger},
  journal={International Mathematics Research Notices},
  volume={2013},
  number={22},
  pages={5196--5226},
  year={2013},
  publisher={OUP}
}

@article{Scott1965-SCOLWD,
	author = {Dana Scott},
	doi = {10.2307/2271545},
	journal = {Journal of Symbolic Logic},
	number = {1},
	pages = {1104--329},
	publisher = {Association for Symbolic Logic},
	title = {Logic with Denumerably Long Formulas and Finite Strings of Quantifiers},
	volume = {36},
	year = {1965}
}

@article{yaacov2017metric,
  title={Metric scott analysis},
  author={{Ben Yaacov}, Ita{\"i} and Doucha, Michal and Nies, Andr{\'e} and Tsankov, Todor},
  journal={Advances in Mathematics},
  volume={318},
  pages={46--87},
  year={2017},
  publisher={Elsevier}
}

@misc{hirvonen2024ehrenfeuchtfraissegamescontinuousfirstorder,
      title={Ehrenfeucht-Fra\"iss\'e Games for Continuous First-Order Logic}, 
      author={\r{A}sa Hirvonen and Joni Puljuj\"arvi},
      year={2024},
      eprint={2402.16662},
      archivePrefix={arXiv},
      primaryClass={math.LO}
}

@article{Farahturbulence2014,
  title={Turbulence, orbit equivalence, and the classification of nuclear C$^*$-algebras},
  author={Farah, Ilijas and Toms, Andrew and T\"ornquist, Asger},
  journal={Journal f\"ur die reine und angewandte Mathematik (Crelles Journal)},
  volume={688},
  pages={101-146},
  year={2014},
}

@book{farah2021MTofC*,
  title={Model Theory of C$^*$-Algebras},
  author={Farah, Ilijas and Hart, Bradd and Lupini, Martino and Robert, Leonel and Tikuisis, Aaron and Vignati, Alessandro and Winter, Wilhelm},
  volume={271},
  year={2021},
  publisher={American Mathematical Society}
}

@book {MR561709,
    AUTHOR = {Moschovakis, Yiannis N.},
     TITLE = {Descriptive set theory},
    SERIES = {Studies in Logic and the Foundations of Mathematics},
    VOLUME = {100},
 PUBLISHER = {North-Holland Publishing Co., Amsterdam-New York},
      YEAR = {1980},
     PAGES = {xii+637},
      ISBN = {0-444-85305-7},
   MRCLASS = {03-02 (03D55 03E15 04A15 28A05 54H05)},
  MRNUMBER = {561709},
MRREVIEWER = {Peter\ G.\ Hinman},
}

@book {MR1321597,
    AUTHOR = {Kechris, Alexander S.},
     TITLE = {Classical descriptive set theory},
    SERIES = {Graduate Texts in Mathematics},
    VOLUME = {156},
 PUBLISHER = {Springer-Verlag, New York},
      YEAR = {1995},
     PAGES = {xviii+402},
      ISBN = {0-387-94374-9},
   MRCLASS = {03E15 (03-01 03-02 04A15 28A05 54H05 90D44)},
  MRNUMBER = {1321597},
MRREVIEWER = {Jakub\ Jasi\'nski},
       DOI = {10.1007/978-1-4612-4190-4},
}

@article {RosenbergSchochetUCT,
    AUTHOR = {Rosenberg, Jonathan and Schochet, Claude},
     TITLE = {The {K}\"unneth theorem and the universal coefficient theorem
              for {K}asparov's generalized {$K$}-functor},
   JOURNAL = {Duke Math. J.},
  FJOURNAL = {Duke Mathematical Journal},
    VOLUME = {55},
      YEAR = {1987},
    NUMBER = {2},
     PAGES = {431--474},
      ISSN = {0012-7094,1547-7398},
   MRCLASS = {46L80 (19K33 46M20 58G12)},
  MRNUMBER = {894590},
MRREVIEWER = {Thierry\ Fack},
       DOI = {10.1215/S0012-7094-87-05524-4},
}

@article {MR3660238,
    AUTHOR = {Lupini, Martino},
     TITLE = {Polish groupoids and functorial complexity},
      NOTE = {With an appendix by Anush Tserunyan},
   JOURNAL = {Trans. Amer. Math. Soc.},
  FJOURNAL = {Transactions of the American Mathematical Society},
    VOLUME = {369},
      YEAR = {2017},
    NUMBER = {9},
     PAGES = {6683--6723},
      ISSN = {0002-9947,1088-6850},
   MRCLASS = {03E15 (22A22 54H05)},
  MRNUMBER = {3660238},
MRREVIEWER = {Zolt\'an\ Vidny\'anszky},
       DOI = {10.1090/tran/7102},
}

@article {MR1057041,
    AUTHOR = {Harrington, Leo and Kechris, Alexander S. and Louveau, Alain},
     TITLE = {A {G}limm-{E}ffros dichotomy for {B}orel equivalence
              relations},
   JOURNAL = {J. Amer. Math. Soc.},
  FJOURNAL = {Journal of the American Mathematical Society},
    VOLUME = {3},
      YEAR = {1990},
    NUMBER = {4},
     PAGES = {903--928},
      ISSN = {0894-0347,1088-6834},
   MRCLASS = {28E15 (03E15 22D40)},
  MRNUMBER = {1057041},
MRREVIEWER = {Gilles\ Godefroy},
       DOI = {10.2307/1990906},
}

@article {MR965754,
    AUTHOR = {Harrington, Leo and Marker, David and Shelah, Saharon},
     TITLE = {Borel orderings},
   JOURNAL = {Trans. Amer. Math. Soc.},
  FJOURNAL = {Transactions of the American Mathematical Society},
    VOLUME = {310},
      YEAR = {1988},
    NUMBER = {1},
     PAGES = {293--302},
      ISSN = {0002-9947,1088-6850},
   MRCLASS = {03E15 (04A15)},
  MRNUMBER = {965754},
MRREVIEWER = {Edward\ Azoff},
       DOI = {10.2307/2001122},
}

@article{friedman1979borel,
  title={Borel structures in mathematics},
  author={Friedman, Harvey},
  journal={manuscript, Ohio State University},
  year={1979}
}

@article {MR2830412,
    AUTHOR = {Hjorth, Greg and Nies, Andr\'e},
     TITLE = {Borel structures and {B}orel theories},
   JOURNAL = {J. Symbolic Logic},
  FJOURNAL = {The Journal of Symbolic Logic},
    VOLUME = {76},
      YEAR = {2011},
    NUMBER = {2},
     PAGES = {461--476},
      ISSN = {0022-4812,1943-5886},
   MRCLASS = {03E15},
  MRNUMBER = {2830412},
MRREVIEWER = {Barbara\ Majcher-Iwanow},
       DOI = {10.2178/jsl/1305810759},
}

@incollection{steinhorn1985chapter,
  title={Chapter XVI: Borel Structures and Measure and Category Logics},
  author={Steinhorn, Charles I},
  booktitle={Model-theoretic logics},
  volume={8},
  pages={579--597},
  year={1985},
  publisher={Association for Symbolic Logic}
}

@article {MR3893282,
    AUTHOR = {Harrison-Trainor, Matthew and Miller, Russell and Montalb\'an,
              Antonio},
     TITLE = {Borel functors and infinitary interpretations},
   JOURNAL = {J. Symb. Log.},
  FJOURNAL = {The Journal of Symbolic Logic},
    VOLUME = {83},
      YEAR = {2018},
    NUMBER = {4},
     PAGES = {1434--1456},
      ISSN = {0022-4812,1943-5886},
   MRCLASS = {03D45 (03C75)},
  MRNUMBER = {3893282},
MRREVIEWER = {Stefan\ Vatev},
       DOI = {10.1017/jsl.2017.81},
}

@incollection {MR2436146,
    AUTHOR = {Ben Yaacov, Ita\"i and Berenstein, Alexander and Henson, C.
              Ward and Usvyatsov, Alexander},
     TITLE = {Model theory for metric structures},
 BOOKTITLE = {Model theory with applications to algebra and analysis. {V}ol.
              2},
    SERIES = {London Math. Soc. Lecture Note Ser.},
    VOLUME = {350},
     PAGES = {315--427},
 PUBLISHER = {Cambridge Univ. Press, Cambridge},
      YEAR = {2008},
      ISBN = {978-0-521-70908-8},
   MRCLASS = {03C65 (28A99 46B99 54E35)},
  MRNUMBER = {2436146},
MRREVIEWER = {Jos\'e\ Iovino},
       DOI = {10.1017/CBO9780511735219.011},
}

@incollection {MR4654490,
    AUTHOR = {Hart, Bradd},
     TITLE = {An introduction to continuous model theory},
 BOOKTITLE = {Model theory of operator algebras},
    SERIES = {De Gruyter Ser. Log. Appl.},
    VOLUME = {11},
     PAGES = {83--131},
 PUBLISHER = {De Gruyter, Berlin},
      YEAR = {2023},
   MRCLASS = {03C66 (03C98 46L05)},
  MRNUMBER = {4654490},
MRREVIEWER = {Shichang\ Song},
}

@article{Ben_Yaacov_2009,
   title={Model theoretic forcing in analysis},
   volume={158},
   ISSN={0168-0072},
   DOI={10.1016/j.apal.2007.10.011},
   number={3},
   journal={Annals of Pure and Applied Logic},
   publisher={Elsevier BV},
   author={Ben Yaacov, Ita\"i and Iovino, Jos\'e},
   year={2009},
   month=apr, pages={163–174} }

@article{Lopez1965,
author = {Lopez-Escobar, E.},
journal = {Fundamenta Mathematicae},
language = {eng},
number = {3},
pages = {253-272},
title = {An interpolation theorem for denumerably long formulas},
volume = {57},
year = {1965},
}

@article{ehrenfeucht1957application,
  title={Application of games to some problems of mathematical logic},
  author={Ehrenfeucht, Andrzej},
  journal={Bull. Acad. Polon. Sci. Cl. III},
  volume={5},
  pages={35--37},
  year={1957}
}

@article{fraisse1955quelques,
  title={Sur quelques classifications des relations, bas{\'e}es sur des isomorphismes restreints},
  author={Fra{\"i}ss{\'e}, Roland},
  journal={Publications Scientifiques de l’Universit{\'e} d’Alger. S{\'e}rie A (math{\'e}matiques)},
  volume={2},
  pages={273--295},
  year={1955}
}

@inproceedings{KARP2014407,
title = {Finite-quantifier equivalence},
author = {Karp, Carol R.},
booktitle = {The Theory of Models, Proceedings of the 1963 International Symposium at Berkeley},
publisher = {North-Holland},
pages = {407-412},
year = {1965},
isbn = {978-0-7204-2233-7},
doi = {https://doi.org/10.1016/B978-0-7204-2233-7.50043-5},
}

@book{hausdorff_grundzuege,
author = {Hausdorff, Felix},
address = {Leipzig},
publisher = {Veit},
title = {Grundz\"uge der Mengenlehre},
year = {1914}
}

@book {chang_keisler_continuous,
    AUTHOR = {Chang, Chen-chung and Keisler, H. Jerome},
     TITLE = {Continuous model theory},
    SERIES = {Annals of Mathematics Studies},
    VOLUME = {No. 58},
 PUBLISHER = {Princeton University Press, Princeton, NJ},
      YEAR = {1966},
     PAGES = {xii+166},
   MRCLASS = {02.50},
  MRNUMBER = {231708},
MRREVIEWER = {G.\ Fuhrken},
}

@article{Ben_Yaacov_2010,
   title={Continuous first order logic and local stability},
   volume={362},
   ISSN={0002-9947},
   DOI={10.1090/s0002-9947-10-04837-3},
   number={10},
   journal={Transactions of the American Mathematical Society},
   publisher={American Mathematical Society (AMS)},
   author={Ben Yaacov, Ita\"i and Usvyatsov, Alexander},
   year={2010},
   month=oct, pages={5213–5213} }

@book{Kechris_2024, place={Cambridge}, series={Cambridge Tracts in Mathematics}, title={The Theory of Countable Borel Equivalence Relations}, publisher={Cambridge University Press}, author={Kechris, Alexander S.}, year={2024}, collection={Cambridge Tracts in Mathematics}}

@article{HIRVONEN2022103123,
title = {Games and Scott sentences for positive distances between metric structures},
journal = {Annals of Pure and Applied Logic},
volume = {173},
number = {7},
pages = {103123},
year = {2022},
issn = {0168-0072},
doi = {https://doi.org/10.1016/j.apal.2022.103123},
author = {\r{A}sa Hirvonen and Joni Puljuj{\" a}rvi}
}

@article{hanson2023approximate,
  title={Approximate isomorphism of metric structures},
  author={Hanson, James E.},
  journal={Mathematical Logic Quarterly},
  volume={69},
  number={4},
  pages={482--507},
  year={2023},
  publisher={Wiley Online Library}
}

\end{document}